\documentclass[a4paper,reqno,12pt]{amsart}
\usepackage{eurosym}
\usepackage{amsfonts}
\usepackage{amsmath}
\usepackage{amssymb}
\usepackage{amsthm}
\usepackage[left=2cm,right=2cm,bottom=3cm,top=3cm]{geometry}
\usepackage[all]{xy}
\usepackage{cite}
\usepackage{bbm}
\usepackage{amscd,color}
\usepackage[hidelinks]{hyperref}

\newtheorem{theorem}{Theorem}[section]
\newtheorem{proposition}[theorem]{Proposition}
\theoremstyle{definition}
\newtheorem{lemma}[theorem]{Lemma}
\newtheorem{corollary}[theorem]{Corollary}

\newtheorem{problem}[theorem]{Problem}

\theoremstyle{definition}
\newtheorem{conjecture}[theorem]{Conjecture}
\newtheorem{example}[theorem]{Example}
\newtheorem{definition}[theorem]{Definition}
\theoremstyle{remark}
\newtheorem{remark}[theorem]{Remark}
\newtheorem{notation}[theorem]{Notation}

\begin{document}
\title[On nonlinear Rado conditions]{On Rado conditions for nonlinear
Diophantine equations}
\author[Jordan Mitchell Barrett]{Jordan Mitchell Barrett}
\address{Jordan Mitchell Barrett\\
School of Mathematics and Statistics\\
Victoria University of Wellington\\
PO Box 600\\
Wellington 6140, New Zealand.}
\email{math@jmbarrett.nz}
\urladdr{http://jmbarrett.nz/}
\author{Martino Lupini}
\address{Martino Lupini\\
School of Mathematics and Statistics\\
Victoria University of Wellington\\
PO Box 600\\
Wellington 6140, New Zealand.}
\email{martino.lupini@vuw.ac.nz}
\urladdr{http://www.lupini.org/}
\author{Joel Moreira}
\address{Joel Moreira\\
Mathematics Institute\\
University of Warwick\\
Coventry, UK.}
\email{joel.moreira@warwick.ac.uk}
\urladdr{https://warwick.ac.uk/fac/sci/maths/people/staff/moreira/}
\thanks{This work was initiated during a visit of J.M. to M.L. at the
California Institute of Technology. The authors gratefully acknowledge the
hospitality and the financial support of the Institute. J.M.B. was supported
by a Summer Research Scholarship from Victoria University of Wellington.
M.L. was partially supported by the NSF Grant DMS-1600186, by a Research
Establishment Grant from Victoria University of Wellington, and by a Marsden
Fund Fast-Start Grant from the Royal Society of New Zealand. J.M. was
partially supported by the NSF Grant DMS-1700147.}
\dedicatory{}
\subjclass[2000]{Primary 05D10, 11D99; Secondary 11U10}
\keywords{Diophantine equation, partition regular, ultrafilter, nonstandard
analysis}

\begin{abstract}
Building on previous work of Di Nasso and Luperi Baglini, we provide general
necessary conditions for a Diophantine equation to be partition regular.
These conditions are inspired by Rado's characterization of partition
regular linear homogeneous equations. We conjecture that these conditions
are also sufficient for partition regularity, at least for equations whose
corresponding monovariate polynomial is linear. This would provide a natural
generalization of Rado's theorem.

We verify that such a conjecture holds for the equations $%
x^{2}-xy+ax+by+cz=0 $ and $x^{2}-y^{2}+ax+by+cz=0$ for $a,b,c\in \mathbb{Z}$
such that $abc=0$ or $a+b+c=0$. To deal with these equations, we establish
new results concerning the partition regularity of polynomial configurations
in $\mathbb{Z}$ such as $\left\{ x,x+y,xy+x+y\right\} $, building on the
recent result on the partition regularity of $\left\{ x,x+y,xy\right\} $.
\end{abstract}

\maketitle


\section{Introduction}

Partition regularity is a central notion in Ramsey theory. Briefly, a family
$\mathcal{A}$ of subsets of $\mathbb{N}$ is \emph{partition regular }if
every finite coloring of $\mathbb{N}$ admits infinitely many monochromatic
sets from $\mathcal{A}$ (see Section \ref{Section:main} for precise
definitions). A \emph{configuration }is partition regular if the family of
sets that realize that configuration is partition regular. For instance, the
seminal Schur Lemma from 1916 \cite{schur_kongruenz_1916} asserts that \emph{%
Schur's triples}, i.e.\ configurations $\left\{ a,b,a+b\right\} $ for $%
a,b\in \mathbb{N}$, are partition regular. Such a result inspired and
motivated one of the early gems of Ramsey theory: the celebrated van der
Waerden theorem on arithmetic progressions \cite{van_der_waerden_beweis_1927}%
, asserting that arithmetic progressions of arbitrary finite length are
partition regular.

Schur's triples can be seen as solutions to the Diophantine equation $z=x+y$%
. Similarly, arithmetic progressions can be seen as solutions to a suitable
system of linear Diophantine equations. This perspective led Rado in 1933
\cite{rado_studien_1933} to provide a complete characterization of which
systems of linear Diophantine equations are partition regular. An equation
(or system of equations) $P\left( x_{1},\ldots ,x_{n}\right) =0$ is called
partition regular if the family
\begin{equation*}
\left\{ \left\{ a_{1},\ldots ,a_{n}\right\} \subseteq \mathbb{N}:P\left(
a_{1},\ldots ,a_{n}\right) =0\right\}
\end{equation*}
of its solutions is partition regular. Rado's theorem provides a
characterization of which linear Diophantine equations are partition regular
in terms of a simple condition on their coefficients (Rado's condition).

Since Rado's theorem, progress in the study of partition regularity of
Diophantine equations has been slower, as the nonlinear case has proved to
be substantially more difficult. Important results were obtained by
Bergelson \cite{bergelson_ergodic_1996,bergelson_ultrafilters_2010},
Khalfalah and Szem\'{e}redi \cite{khalfalah_number_2006}, Csikv\'{a}ri,
Gyarmati, and S\'{a}rk\H{o}zy \cite{csikvari_density_2012}, Hindman \cite%
{hindman_monochromatic_2011}, Frantzikinakis and Host \cite%
{frantzikinakis_higher_2017}, Green and Sanders \cite%
{green_monochromatic_2016}, Bergelson and Moreira \cite%
{bergelson_measure_2018,bergelson_ergodic_2017}, Moreira \cite%
{moreira_monochromatic_2017}, and Bergelson, Moreira, and Johnson \cite%
{bergelson_new_2017}. Recently, some general necessary conditions for
partition regularity were obtained using nonstandard methods by Di Nasso and
Luperi Baglini \cite{di_nasso_ramsey_2018}, building on previous work of Di
Nasso and Riggio \cite{di_nasso_fermat-like_2018}, and Luperi Baglini \cite%
{luperi_baglini_partition_2014,luperi_baglini_hyperintegers_2012}.

Despite these efforts, the problem of determining which Diophantine
equations are partition regular has been solved only in very special cases.
A major problem in this area is to find explicit necessary and sufficient
conditions for a Diophantine equation to be partition regular, generalizing
Rado's theorem in the linear case. Incidentally, it should be noted that the
problem of determining whether a Diophantine equation has \emph{some}
integer solution is undecidable (in the sense of computability theory), as
shown by Yuri Matiyasevich---building on previous work of Martin Davis,
Hilary Putnam, and Julia Robinson---in his solution of Hilbert's 10th
problem \cite{matiyasevich_enumerable_1970}.

In this work, we present general necessary conditions for partition
regularity of Diophantine equations. We name them \emph{Rado conditions}, as
they are inspired by Rado's characterization of partition regular linear
Diophantine equations, and indeed in the particular case of linear
equations, they recover the conditions from Rado's theorem. They can also be
seen as a strengthening of some conditions recently proved to be necessary
by Di Nasso and Luperi Baglini \cite{di_nasso_ramsey_2018}; see also \cite%
{di_nasso_nonstandard_2019}. We show that the Rado conditions are indeed
\emph{sufficient }for partition regularity for the equations $%
x^{2}-xy+ax+by+cz=0$ and $x^{2}-y^{2}+ax+by+cz=0$ for $a,b,c\in \mathbb{Z}$
such that $abc=0$ or $a+b+c=0$.

In order to verify that the Rado conditions are sufficient for the latter
equation, we parametrize the solutions with polynomial configurations in $%
\mathbb{Z}$ of the form
\begin{equation*}
\left\{ r+p\left( s\right) ,r+q\left( s\right) ,rs+r+ds\right\} \text{,}
\end{equation*}%
where $p\left( s\right) ,q\left( s\right) $ are polynomials with integer
coefficients vanishing at $0$, and $d\in \mathbb{Q}$. We show that such
configurations are partition regular, building on the analogue result from
\cite{moreira_monochromatic_2017} for the configurations $\left\{ r+p\left(
s\right) ,r+q\left( s\right) ,rs\right\} $.

We believe that the results of this paper are an important contributions to
the study of partition regularity of Diophantine equations. Indeed, they
provide general necessary conditions that can be used to rule out partition
regularity in many cases, and can be seen as a natural generalization to the
nonlinear case of the condition in Rado's theorem. This paper is divided in
five sections, besides this introduction. Section \ref{Section:main}
introduces the terminology to be used in the rest of the paper, and presents
the statements of the main results. In\ Section \ref{Section:necessity} we
obtain the necessity of the Rado conditions, and in Section \ref%
{Section:sufficient} we establish that certain polynomial configurations in $%
\mathbb{Z}$ are partition regular. These results are then applied in Section %
\ref{Section:equation} to show that the Rado conditions are both necessary
and sufficient in the case of the equations $x^{2}-xy+ax+by+cz=0$ and $%
x^{2}-y^{2}+ax+by+cz=0$ where $a,b,c\in \mathbb{Z}$ are such that $a+b+c=0$
or $abc=0$.

\section{Main results\label{Section:main}}

We let $\mathbb{N}$ be the set of \emph{strictly positive }integers, and set
$\mathbb{N}_{0}:=\mathbb{N}\cup \left\{ 0\right\} $. A finite coloring of $%
\mathbb{N}$ is a function $c:\mathbb{N}\rightarrow \left\{ 1,2,\ldots
,k\right\} $ for some $k\in \mathbb{N}$. In this case, $k$ is the number of
colors, and the elements of $\left\{ 1,2,\ldots ,k\right\} $ are the \emph{%
colors}. A subset $A$ of $\mathbb{N}$ is \emph{monochromatic }for the
coloring $c$ if there exists a color $i\in \left\{ 1,2,\ldots ,k\right\} $
such that, for every $a\in A$, $c\left( a\right) =i$.

\begin{definition}
Let $\mathcal{A}$ be a family of subsets of $\mathbb{N}$. We say that $%
\mathcal{A}$ is partition regular if, for every finite coloring $c$ of $%
\mathbb{N}$, there exist infinitely many $A\in \mathcal{A}$ that are
monochromatic for $c$.
\end{definition}

In the following we will use the same notation and terminology for
polynomials as in \cite[Section 3]{di_nasso_ramsey_2018}. Particularly, a
(multi)index is an $n$-tuple $\alpha =\left( \alpha _{1},\ldots ,\alpha
_{n}\right) \in \mathbb{N}_{0}^{n}$. We define the degree of $\alpha $ to be
$\left\vert \alpha \right\vert :=\alpha _{1}+\cdots +\alpha _{n}$, and the
(Hamming) length $\ell \left( \alpha \right) $ to be $\left\vert \left\{
i\in \left\{ 1,2,\ldots ,n\right\} :\alpha _{i}>0\right\} \right\vert $. If $%
x=\left( x_{1},\ldots ,x_{n}\right) $ is an $n$-tuple of variables, $%
x^{\alpha }$ represents the monomial $x_{1}^{\alpha _{1}}\cdots
x_{n}^{\alpha _{n}}$. We can then represent a polynomial $P\in \mathbb{Z}%
\left[ x_{1},\ldots ,x_{n}\right] $ in the variables $x_{1},\ldots ,x_{n}$
as $\sum_{\alpha }c_{\alpha }x^{\alpha }$, where $\alpha $ ranges in $%
\mathbb{N}_{0}^{n}$ and $c_\alpha=0$ for all but finitely many $\alpha$. We
define $\mathrm{Supp}\left( P\right) \subset \mathbb{N}_{0}^{n}$ to be the
(finite) set of indices $\alpha $ such that $c_{\alpha }$ is nonzero. There
is a natural partial order among indices, obtained by setting $\alpha \leq
\beta $ if and only if $\alpha _{i}\leq \beta _{i}$ for every $i\in \left\{
1,2,\ldots ,n\right\} $. An index $\alpha $ is \emph{maximal }for $P$ if it
is a maximal element in $\mathrm{Supp}\left( P\right) $ with respect to such
an order, and \emph{minimal }for $P$ if it is a minimal element in $\mathrm{%
Supp}\left( P\right) $. Accordingly, the term $c_{\alpha }x^{\alpha }$ in $P$
is maximal (respectively, minimal) if $\alpha $ is a maximal (respectively,
minimal) index for $P$. We say that a set $J$ of indices is homogeneous of
degree $d$ if $\left\vert \alpha \right\vert =d$ for every $\alpha \in J$.
Furthermore, for indices $\beta \leq \alpha $ we set%
\begin{equation*}
\alpha !=\prod_{i=1}^{n}(\alpha _{i}!)\qquad \text{and}\qquad \binom{\alpha
}{\beta }=\prod_{i=1}^{n}\binom{\alpha _{i}}{\beta _{i}}=\frac{\alpha !}{%
\beta !(\alpha -\beta )!}\text{.}
\end{equation*}%
%
%
%
%
%
%
%
%
%
%
%
%
%
%
%
%
%
%
%
%
%
%
For $P\in \mathbb{Z}\left[ x_{1},\ldots ,x_{n}\right] $ we define the
partial derivatives%
\begin{equation*}
\frac{\partial ^{\beta }P}{\partial x^{\beta }}:=\frac{\partial ^{\left\vert
\beta \right\vert }P}{\partial x_{1}^{\beta _{1}}\partial x_{2}^{\beta
_{2}}\cdots \partial x_{n}^{\beta _{n}}}\text{.}
\end{equation*}%
Note that
\begin{equation*}
\frac{\partial ^{\beta }x^{\alpha }}{\partial x^{\beta }}=\beta !\binom{%
\alpha }{\beta }x^{\alpha -\beta }
\end{equation*}
if $\beta \leq \alpha $ and it equals $0$ otherwise. For $r\in \mathbb{Z}$,
we consider the\emph{\ }polynomial $P^{\left( r\right) }\left( x_{1},\ldots
,x_{n}\right) =P(x_{1}+r,\ldots ,x_{n}+r)$. One can express the coefficients
of $P^{\left( r\right) }\left( x\right) $ in terms of the partial
derivatives of $P$ evaluated at $r$, as the following computation shows.
Suppose that $P\left( x\right) =\sum_{\alpha }c_{\alpha }x^{\alpha }$. Then
we have that%
\begin{eqnarray*}
P^{(r)}\left( x\right) &=&\sum_{\alpha }c_{\alpha }\left( x_{1}+r\right)
^{\alpha _{1}}\cdots \left( x_{n}+r\right) ^{\alpha _{n}} \\
&=&\sum_{\alpha }\sum_{\beta \leq \alpha }c_{\alpha }\binom{\alpha }{\beta }%
r^{\left\vert \alpha \right\vert -\left\vert \beta \right\vert }x^{\beta } \\
&=&\sum_{\beta }(\sum_{\alpha \geq \beta }c_{\alpha }\binom{\alpha }{\beta }%
r^{\left\vert \alpha \right\vert -\left\vert \beta \right\vert })x^{\beta }
\\
&=&\sum_{\beta }\frac{1}{\beta !}\frac{\partial ^{\beta }P}{\partial
x^{\beta }}\left( r,\ldots ,r\right) x^{\beta }\text{.}
\end{eqnarray*}%
If we let $\boldsymbol{r}$ be the $n$-tuple $\left( r,\ldots ,r\right) $,
then we can write%
\begin{equation*}
P^{\left( r\right) }\left( x\right) =\sum_{\beta }\frac{1}{\beta !}\frac{%
\partial ^{\beta }P}{\partial x^{\beta }}\left( \boldsymbol{r}\right)
x^{\beta }\text{.}
\end{equation*}%
In particular, one has that%
\begin{equation*}
P\left( x\right) =\sum_{\beta }\frac{1}{\beta !}\frac{\partial ^{\beta }P}{%
\partial x^{\beta }}\left( \boldsymbol{0}\right) x^{\beta }\text{.}
\end{equation*}%
We denote by $\tilde{P}\left( w\right) \in \mathbb{Z}\left[ w\right] $ the
monovariate polynomial $\tilde{P}\left( w\right) =P\left( w,\ldots ,w\right)
$.

For a prime $p\in \mathbb{N}$, we denote by $\mathbb{Z}_{p}$ the ring of $p$%
-adic integers; see \cite[Chapter 1, Section 1]{robert_course_2000}. It is
easy to see that, for $Q\left( w\right) \in \mathbb{Z}\left[ w\right] $, the
equation $Q\left( w\right) =0$ has an invertible solution in $\mathbb{Z}_{p}$
if and only if, for every $n\in \mathbb{N}$, the equation $Q\left( w\right)
=0$ has an invertible solution in $\mathbb{Z}/p^{n}\mathbb{Z}$; see \cite[%
Chapter 1, Section 6]{robert_course_2000}.

\begin{definition}
Fix $q\in \mathbb{N}$ with $q\geq 2$. Consider a function $f:\mathbb{N}%
^{n}\rightarrow \mathbb{Q}$. Define
\begin{equation*}
\mathcal{A}_{f}=\big\{\{a_{1},\ldots ,a_{n}\}\subseteq \mathbb{N}:f\left(
a_{1},\ldots ,a_{n}\right) =0\big\}\text{.}
\end{equation*}%
The equation $f\left( x_{1},\ldots ,x_{n}\right) =0$, and the function $f$,
are called:

\begin{itemize}
\item partition regular if the family $\mathcal{A}_{f}$ is partition regular;

\item $q$-partition regular if, for every $k\in \mathbb{N}$, the family
consisting of $\left\{ a_{1},\ldots ,a_{n}\right\} \in \mathcal{A}_{f}$ such
that $a_{1}\equiv \cdots \equiv a_{n}\equiv 0\mathrm{\ \mathrm{mod}}\ q^{k}$%
, is partition regular.
\end{itemize}

In particular, this definition applies when the function $f$ is a polynomial
with integer coefficients.
\end{definition}

It is clear that a homogeneous polynomial is partition regular if and only
if it is $q$-partition regular for every $q>1$.
We remark in passing that the converse does not hold, i.e., there are
polynomials (e.g. $f(x,y,z)=xy-z$) which are $q$-partition regular for every
$q>1$ but are not homogenous. Nevertheless, a $q$-partition regular equation
necessarily has zero constant term. On the other hand, a partition regular
equation might have a nonzero constant term, such as $x-y+z^{2}=1$ (this
follows from \cite[Theorem 5.2]{Bergelson_Leibman_Lesigne08} with $r=1$ and $%
p_{1}(z)=z^{2}-1$).


\begin{proposition}
\label{Proposition:scale}If the equation $P\left( x_{1},\ldots ,x_{n}\right)
=0$ is partition regular, then for every $r\in \mathbb{Z}\setminus \left\{
0\right\} $, the equations $P\left( x_{1}+r,\ldots ,x_{n}+r\right) =0$ and $%
P\left( \frac{1}{r}x_{1},\ldots ,\frac{1}{r}x_{n}\right) =0$ are also
partition regular. If the equation $P\left( x_{1},\ldots ,x_{n}\right) =0$
is $q$-partition regular for some $q\in \mathbb{N}\setminus \left\{
1\right\} $, then the equation $P\left(qx_{1},\ldots ,qx\right) =0$ is also $%
q$-partition regular.
\end{proposition}

\begin{proof}
We prove only the last assertion; the others can be established in a similar
way. Suppose $c:\mathbb{N}\rightarrow \{1,\dots ,k\}$ is an arbitrary
coloring. Consider a new coloring $\tilde{c}:\mathbb{N}\to\{1,\dots,k+1\}$
defined as $\tilde{c}(i)=c(i/q) $ if $i\equiv 0\bmod q$ and $\tilde{c}%
(i)=k+1 $ otherwise. Let $(a_{1},\dots ,a_{n})\in \mathbb{N}^{n}$ be a $%
\tilde{c}$-monochromatic solution to $P(x_{1},\dots ,x_{n})=0$ with all $%
a_{i}$ multiples of $q$. Then $(\frac{1}{q}a_{1},\dots ,\frac{1}{q}a_{n})$
is a $c$-monochromatic solution to $P\left( qx_{1},\ldots ,qx_{n}\right) =0$.
\end{proof}

The classical characterization of partition regularity for homogeneous \emph{%
linear }polynomials due to Rado \cite{rado_studien_1933} states that a
(homogeneous) linear polynomial $P(x)=a_1x_1+\cdots+a_nx_n$ is partition
regular if and only if there exists a set $J\subseteq\{1,\dots,n\}$ such
that $\sum_{j\in J}a_j=0$. This condition can be phrased without explicitly
using the coefficients of $P$ by saying that there exists a set $J\subseteq%
\mathrm{Supp}\left( P\right)$ such that
\begin{equation}  \label{eq_radooriginal}
\sum_{\alpha \in J}\frac1{\alpha!}\frac{\partial ^{\alpha }P}{\partial
x^{\alpha }}\left( \boldsymbol{0}\right) =0.
\end{equation}

For instance, the equation $x+y=3z$ is not partition regular. Rado's
condition can also be used to deduce that certain polynomial equations of
higher degree are not partition regular. We illustrate this point in the
following examples.

\begin{example}
\label{Example:homogeneousRado} The equation $x^2+y^2=3z^2$ is not partition
regular.

Indeed, given a coloring $c:\mathbb{N}\to\{1,\dots,k\}$, let $\tilde c:%
\mathbb{N}\to\{1,\dots,k\}$ be the coloring defined by $\tilde c(x)=c(x^2)$.
Whenever $(x,y,z)$ is a $\tilde c$-monochromatic solution to $x^2+y^2=3z^2$,
the triple $(x^2,y^2,z^2)$ is a $c$-monochromatic solution to $x+y=3z$.
Choosing a coloring $c$ which admits no monochromatic solution to $x+y=3z$
we constructed a coloring $\tilde c$ which admits no monochromatic solutions
to $x^2+y^2=3z^2$.
\end{example}

In fact, a variation of this idea can be extended to work for general
homogeneous polynomials. This observation was first made by Lefmann in \cite[%
Theorem 2.1]{Lefmann91}, and was also used in \cite[Corollary 3.9]%
{di_nasso_ramsey_2018}.

\begin{example}
\label{Example:multiplicativeRado} The equation $xy=z^3$ is not partition
regular.

Indeed, let $f:\mathbb{N}\to\mathbb{N}$ be a map satisfying $f(xy)=f(x)+f(y)$
for all $x,y\geq2$. Given a coloring $c:\mathbb{N}\to\{1,\dots,k\}$, let $%
\tilde c:\mathbb{N}\to\{1,\dots,k\}$ be the coloring defined by $\tilde
c(x)=c\big(f(x)\big)$. Whenever $(x,y,z)$ is a $\tilde c$-monochromatic
solution to $xy=z^3$, the triple $(f(x),f(y),f(z))$ is a $c$-monochromatic
solution to $x+y=3z$. Choosing a coloring $c$ which admits no monochromatic
solution to $x+y=3z$ we constructed a coloring $\tilde c$ which admits no
monochromatic solutions to $xy=z^3$ other than the trivial solution $(1,1,1)$%
.
\end{example}

The idea used in Example \ref{Example:multiplicativeRado} can also be
adapted to work for more general monomials. In fact, it is not necessary to
assume that the map $f$ satisfies $f(xy)=f(x)+f(y)$: it suffices that $%
f(xy)-f(x)-f(y)$ is bounded. This will be the case when $f(x)$ is the number
of digits of $x$ in base $q$ (for an arbitrary fixed $q$). The advantage of
using such $f$ is that the original proof of the necessity of the Rado
condition also involves a coloring which looks at the base $q$ expansion of $%
x$.

In what follows, we combine the two methods from Examples \ref%
{Example:homogeneousRado} and \ref{Example:multiplicativeRado} to generate
new necessary conditions for polynomial equations of higher degree to be
partition regular.
Before doing so we need some more definitions.

\begin{definition}
Let $\phi :\mathbb{Z}^{n}\rightarrow \mathbb{Z}$ be a positive linear map $%
\left( \alpha _{1},\ldots ,\alpha _{n}\right) \mapsto t_{1}\alpha
_{1}+\cdots +t_{n}\alpha _{n}$ for some $t_{1},\ldots ,t_{n}\in \mathbb{N}%
_{0}$.

\begin{itemize}
\item If $c$ is a finite coloring of $\mathbb{N}_{0}$, then we say that $%
\phi $ is $c$-monochromatic if $\left\{ t_{1},\ldots ,t_{n}\right\} $ is $c$%
-monochromatic;

\item If $P\in \mathbb{Z}\left[ x_{1},\ldots ,x_{n}\right] $, and $\left(
M_{0},\ldots ,M_{\ell }\right) $ is the increasing enumeration of $\phi
\left( \mathrm{Supp}\left( P\right) \right) $, then the partition of $%
\mathrm{Supp}\left( P\right) $ determined by $\phi $ is the ordered tuple $%
\left( J_{0},\ldots ,J_{\ell }\right) $, where $J_{i}=\left\{ \alpha \in
\mathrm{\mathrm{Supp}}\left( P\right) :\phi \left( \alpha \right)
=M_{i}\right\} $.
\end{itemize}
\end{definition}

Suppose that $P\in \mathbb{Z}\left[ x_{1},\ldots ,x_{n}\right] $ is a
polynomial.

\begin{definition}
A \emph{Rado partition} of $P$ is an ordered tuple $\left( J_{0},\ldots
,J_{\ell }\right) $ such that, for every finite coloring $c$ of $\mathbb{N}%
_{0}$, there exist infinitely many $c$-monochromatic positive linear maps $%
\phi :\mathbb{Z}^{n}\rightarrow \mathbb{Z}$ such that $\left( J_{0},\ldots
,J_{\ell }\right) $ is the partition of $\mathrm{Supp}\left( P\right) $
determined by $\phi $.
\end{definition}

Since $\phi$ must be a positive linear map, $J_\ell$ contains at least one
maximal index, and $J_0$ contains at least one minimal index.

\begin{definition}
A \emph{Rado set }for $P$ is a set $J\subseteq \mathrm{Supp}\left( P\right) $
such that there exists a Rado partition $\left( J_{0},\ldots ,J_{\ell
}\right) $ for $P$ such that $J=J_{i}$ for some $i\in \left\{ 0,1,\ldots
,\ell \right\} $.

A minimal (respectively, maximal) Rado set\emph{\ }for $P$ is a set $%
J\subseteq \mathrm{Supp}\left( P\right) $ such that there exists a Rado
partition $\left( J_{0},\ldots ,J_{\ell }\right) $ for $P$ with $J=J_{0} $
(respectively, $J=J_{\ell }$).
\end{definition}


\begin{example}
\label{Example:Rado-partition}Consider the polynomial $P\left(
x_1,x_2,x_3,x_4\right) =-x_{3}-x_{4}+x_{1}x_{2}+x_{1}x_{2}^{2}$. We have that%
\begin{equation*}
\mathrm{Supp}\left( P\right) =\left\{ \left( 0,0,1,0\right) ,\left(
0,0,0,1\right) ,\left( 1,1,0,0\right) ,\left( 1,2,0,0\right) \right\} \text{.%
}
\end{equation*}%
We claim that the partition $J=\left( J_{0},J_{1},J_{2},J_{3}\right) $
defined by%
\begin{eqnarray*}
J_{0} &=&\left\{ \left( 0,0,1,0\right) \right\} \\
J_{1} &=&\left\{ \left( 1,1,0,0\right) \right\} \\
J_{2} &=&\left\{ \left( 0,0,0,1\right) \right\} \\
J_{3} &=&\left\{ \left( 1,2,0,0\right) \right\}
\end{eqnarray*}%
is a Rado partition for $P$. Indeed, suppose that $c$ is a finite coloring
of $\mathbb{N}$. By Brauer's extension of the van der Waerden theorem on
arithmetic progressions, there exist $a,b\geq 3$ such that $\left\{
a,b,a+b-1,a+2b-2\right\} $ is $c$-monochromatic. We can consider the
corresponding $c$-monochromatic linear map $\phi :\mathbb{Z}^{4}\rightarrow
\mathbb{Z}$,
\begin{equation*}
\left( \alpha _{1},\alpha _{2},\alpha _{3},\alpha _{4}\right) \mapsto
a\alpha _{1}+b\alpha _{2}+\left( a+b-1\right) \alpha _{3}+\left(
a+2b-2\right) \alpha _{4}\text{.}
\end{equation*}%
Then%
\begin{equation*}
\phi \left( 0,0,1,0\right) <\phi \left( 1,1,0,0\right) <\phi \left(
0,0,0,1\right) <\phi \left( 1,2,0,0\right) \text{.}
\end{equation*}%
Therefore the partition of $\mathrm{Supp}\left( P\right) $ determined by $%
\phi $ is $J$.
\end{example}

\begin{notation}
Consider variables $z_{1},\ldots ,z_{n}$. For $\alpha \in \mathbb{Z}^{n}$
let $p_{\alpha }\left( z_{1},\ldots ,z_{n}\right) $ be the linear polynomial
$\alpha _{1}z_{1}+\cdots +\alpha _{n}z_{n}$. For $J\subseteq \mathbb{N}%
_{0}^{n}$, let $E_{J}$ be the $\mathbb{Q}$-linear span of $\left\{ p_{\alpha
-\beta }:\alpha ,\beta \in J\right\} $.
\end{notation}

\begin{proposition}
Suppose that $\left( J_{0},\ldots ,J_{\ell }\right) $ is a Rado partition
for $P$. Then one has that, for every $i\in \left\{ 0,1,\ldots ,\ell
\right\} $:

\begin{enumerate}
\item The set $J_i$ is convex, in the sense that if $\alpha^{(1)},\ldots
,\alpha^{(m)}\in J_{i}$ and $\lambda _{1},\ldots ,\lambda _{m}\in \mathbb{Q}$
are such that $\lambda _{1}+\cdots +\lambda _{m}=1$, and $\lambda
_{1}\alpha^{(1)}+\cdots +\lambda _{m}\alpha^{(m)}\in\mathrm{Supp}\left(
P\right) $, then $\lambda _{1}\alpha^{(1)}+\cdots +\lambda
_{m}\alpha^{(m)}\in J_{i}$;

\item every nonzero polynomial $p\left( z_{1},\ldots ,z_{n}\right)
=a_{1}z_{1}+\cdots +a_{n}z_{n}$ in $E_{J_{i}}$ is partition regular, i.e.\
there exist $\ell \geq 2$ and $1\leq i_{1}<i_{2}<\cdots <i_{\ell }\leq n$
such that $a_{i_{1}},\ldots ,a_{i_{\ell }}$ are nonzero and $%
a_{i_{1}}+\cdots +a_{i_{\ell }}=0$;

\item if $\alpha ,\beta \in J_{i}$ and $\ell \left( \alpha +\beta \right)
\leq 2$, then $\left\vert \alpha \right\vert =\left\vert \beta \right\vert $.
\end{enumerate}
\end{proposition}

\begin{proof}
\begin{enumerate}
\item This follows from the definitions and the fact that linear maps
preserve affine combinations.

\item By the definition of Rado partition, for every finite coloring $c$ of $%
\mathbb{N}$ there exists a positive linear map $\phi:\mathbb{Z}^n\to\mathbb{Z%
}$, say $\phi:(\alpha_1,\dots,\alpha_n)\mapsto
t_1\alpha_1+\cdots+t_n\alpha_n $, such that $\{t_1,\dots,t_n\}$ is $c$%
-monochromatic and $\phi(\alpha)=\phi(\beta)$ for every $\alpha,\beta\in J_i$%
. Therefore $p_{\alpha-\beta}(t_1,\dots,t_n)=0$ whenever $\alpha,\beta\in
J_i $, and hence $p(t_1,\dots,t_n)=0$ whenever $p\in E_{J_i}$.

\item If $\ell(\alpha+\beta)\leq2$, then the polynomial $p_{\alpha-\beta}$
depends on at most two variables, and hence if it is partition regular must
take the form $p_{\alpha-\beta}(z_1,\dots,z_n)=\lambda(z_i-z_j)$ for some $%
\lambda\in\mathbb{N}_0$ and some $i\neq j$. Consequently, $%
|\alpha|-|\beta|=p_{\alpha-\beta}(1,\dots,1)=0$ as required.
\end{enumerate}
\end{proof}

\begin{definition}
A\emph{\ lower Rado functional }of order $m\in \mathbb{N}_{0}$ for $P\in
\mathbb{Z}\left[ x_{1},\ldots ,x_{n}\right] $ is a tuple $\left(
J_{0},\ldots ,J_{\ell },d_{1},\ldots ,d_{m}\right) $ for some $\ell \geq m$
and $d_{1},\ldots ,d_{m}\in \mathbb{N}$ such that, for every finite coloring
$c$ of $\mathbb{N}$ and for every $n\in \mathbb{N}$, there exist infinitely
many $c$-monochromatic positive linear maps $\phi :\mathbb{Z}^{n}\rightarrow
\mathbb{Z}$, $\left( \alpha _{1},\ldots ,\alpha _{n}\right) \mapsto
t_{1}\alpha _{1}+\cdots +t_{n}\alpha _{n}$ such that $\left( J_{0},\ldots
,J_{\ell }\right) $ is the partition of $\mathrm{Supp}\left( P\right) $
determined by $\phi $ and, if $\left( M_{0},\ldots ,M_{\ell }\right) $ is
the increasing enumeration of $\phi \left( \mathrm{Supp}\left( P\right)
\right) $, then $M_{i}-M_{0}=d_{i}$ for $i\in \left\{ 1,2,\ldots ,m\right\} $%
, and $M_{m+1}-M_{m}\geq n$.
\end{definition}

\begin{definition}
An\emph{\ upper Rado functional }of order $m\in \mathbb{N}_{0}$ for $P\in
\mathbb{Z}\left[ x_{1},\ldots ,x_{n}\right] $ is a tuple $\left(
J_{0},\ldots ,J_{\ell },d_{0},\ldots ,d_{m-1}\right) $ for some $\ell \geq m$
and $d_{0},\ldots ,d_{m-1}\in \mathbb{N}$ such that, for every finite
coloring $c$ of $\mathbb{N}$ and for every $n\in \mathbb{N}$, there exist
infinitely many $c$-monochromatic positive linear maps $\phi :\mathbb{Z}%
^{n}\rightarrow \mathbb{Z}$, $\left( \alpha _{1},\ldots ,\alpha _{n}\right)
\mapsto t_{1}\alpha _{1}+\cdots +t_{n}\alpha _{n}$ such that $\left( J_{\ell
},\ldots ,J_{0}\right) $ is the partition of $\mathrm{Supp}\left( P\right) $
determined by $\phi $ and, if $\left( M_{\ell },\ldots ,M_{0}\right) $ is
the increasing enumeration of $\phi \left( \mathrm{Supp}\left( P\right)
\right) $, then $M_{i}-M_{m}=d_{i}$ for $i\in \left\{ 0,1,\ldots
,m-1\right\} $, and $M_{m}-M_{m+1}\geq n$.
\end{definition}

In Remark \ref{Remark:partition} below we give a simpler description of Rado
functionals in the language of non-standard analysis.

\begin{example}
Consider as in Example \ref{Example:Rado-partition} the polynomial $P\left(
x_{1},x_{2},x_{3},x_{4}\right) =-x_{3}-x_{4}+x_{1}x_{2}+x_{1}x_{2}^{2}$ and
the Rado partition $\left( J_{0},J_{1},J_{2},J_{3}\right) $ defined there.
We claim that $\left( J_{3},J_{2},J_{1},J_{0},2\right) $ is an upper Rado
functional of order $1$ for $P$, and that $\left(
J_{0},J_{1},J_{2},J_{3},1\right) $ is a lower Rado functional of order $1$
for $P$. Indeed, suppose that $c$ is a finite coloring of $\mathbb{N}$, and $%
n\in \mathbb{N}$. By Brauer's extension of the van der Waerden theorem on
arithmetic progressions there exist $a,b\geq n+2$ such that $\left\{
a,b,a+b-1,a+2b-2\right\} $ is $c$-monochromatic. Consider the linear
functional $\phi :\mathbb{Z}^{4}\rightarrow \mathbb{Z}$,%
\begin{equation*}
\left( \alpha _{1},\alpha _{2},\alpha _{3},\alpha _{4}\right) \mapsto
a\alpha _{1}+b\alpha _{2}+\left( a+b-1\right) \alpha _{3}+\left(
a+2b-2\right) \alpha _{4}\text{.}
\end{equation*}%
Then we have that%
\begin{equation*}
\phi \left( 0,0,1,0\right) <\phi \left( 1,1,0,0\right) <\phi \left(
0,0,0,1\right) <\phi \left( 1,2,0,0\right) \text{.}
\end{equation*}%
Furthermore,%
\begin{eqnarray*}
\phi \left( 1,2,0,0\right) -\phi \left( 0,0,0,1\right) &=&2 \\
\phi \left( 0,0,0,1\right) -\phi \left( 1,1,0,0\right) &=&b-2\geq n \\
\phi \left( 1,1,0,0\right) -\phi \left( 0,0,1,0\right) &=&1\text{.}
\end{eqnarray*}%
This witnesses that $\left( J_{3},J_{2},J_{1},J_{0},2\right) $ is an upper
Rado functional of order $1$ for $P$, and that $\left(
J_{0},J_{1},J_{2},J_{3},1\right) $ is a lower Rado functional of order $1$
for $P$.
\end{example}

\begin{example}
Consider the polynomial $P\left( x,y,z,w\right) =xy^{2}+yz+w$. Then we have
that%
\begin{equation*}
\mathrm{Supp}\left( P\right) =\left\{ \left( 1,2,0,0\right) ,\left(
0,1,1,0\right) ,\left( 0,0,0,1\right) \right\} \text{.}
\end{equation*}%
We claim that $\left( J_{0},J_{1},J_{2},1,2\right) $ is a lower Rado
functional of order $2$ for $P$, where%
\begin{eqnarray*}
J_{0} &=&\left\{ \left( 0,0,0,1\right) \right\} \\
J_{1} &=&\left\{ \left( 0,1,1,0\right) \right\} \\
J_{2} &=&\left\{ \left( 1,2,0,0\right) \right\} \text{.}
\end{eqnarray*}%
Indeed, suppose that $c$ is a finite coloring of $\mathbb{N}$. By Brauer's
extension of the van der Waerden theorem there exist $a,b\geq n+2$ such that
$\left\{ a,b,a+b-1,a+2b-2\right\} $ is $c$-monochromatic. Consider the
linear functional $\phi :\mathbb{Z}^{4}\rightarrow \mathbb{Z}$,%
\begin{equation*}
\left( \alpha _{1},\alpha _{2},\alpha _{3},\alpha _{4}\right) \mapsto
a\alpha _{1}+b\alpha _{2}+\left( a+b-1\right) \alpha _{2}+\left(
a+2b-2\right) \alpha _{3}\text{.}
\end{equation*}%
Then we have that%
\begin{equation*}
\phi \left( 0,0,0,1\right) +2=\phi \left( 0,1,1,0\right) +1=\phi \left(
1,2,0,0\right) =a+2b\text{.}
\end{equation*}%
This witnesses that $\left( J_{0},J_{1},J_{2},1,2\right) $ is a lower Rado
functional of order $2$ for $P$.
\end{example}

We consider the following generalizations of Rado's criterion of partition
regularity.

\begin{definition}
\label{Definition:maximal-Rado}A polynomial $P\left( x\right) =\sum_{\alpha
}c_{\alpha }x^{\alpha }\in \mathbb{Z}\left[ x_{1},\ldots ,x_{n}\right] $
satisfies the\emph{\ maximal Rado condition }if for every $q\in \mathbb{N}%
\setminus \left\{ 1\right\} $ there exists an upper Rado functional $\left(
J_{0},\ldots ,J_{\ell },d_{0},d_{1},\ldots ,d_{m-1}\right) $ for $P$ such
that, setting $d_{m}:=0$, the polynomial
\begin{equation*}
w\mapsto\sum_{i=0}^{m}q^{d_{i}}\sum_{\alpha \in J_{i}}c_{\alpha
}w^{\left\vert \alpha \right\vert }
\end{equation*}%
has a root in $\left[ 1,q\right] $.
\end{definition}

\begin{definition}
\label{Definition:minimal-Rado}Let $P$ be a polynomial such that $\tilde{P}$
is nonzero and splits as a product of linear factors over $\mathbb{Z}$. We
say that $P$ satisfies the \emph{minimal Rado condition} if for every prime $%
p\in \mathbb{N}$ there exists a root $a\in \mathbb{Z}$ of $\tilde{P}$ and a
minimal Rado functional $\left( J_{0},\ldots ,J_{\ell },d_{1},\ldots
,d_{m}\right) $ for $P^{\left( a\right) }$ such that, setting $d_{0}:=0$,
the equation
\begin{equation*}
\sum_{i=0}^{m}p^{d_{i}}\sum_{\alpha \in J_{i}}\frac{1}{\alpha !}\frac{%
\partial ^{\alpha }P}{\partial x^{\alpha }}\left( \boldsymbol{a}\right)
w^{\left\vert \alpha \right\vert }=0
\end{equation*}%
has an invertible solution in $\mathbb{Z}_{p}$.
\end{definition}

\begin{lemma}
\label{Lemma:minimal-Rado}Let $P$ be a polynomial such that $\tilde{P}$ is
nonzero and splits as a product of linear factors over $\mathbb{Z}$ and $P$
satisfies the minimal Rado condition.
Then there exists a root $a\in \mathbb{%
Z}$ of $\tilde{P}$ and a minimal Rado set $J$ of $P^{\left( a\right) }$ such
that the equation%
\begin{equation*}
\sum_{\alpha \in J}\frac{1}{\alpha !}\frac{\partial ^{\alpha }P}{\partial
x^{\alpha }}\left( \boldsymbol{a}\right) w^{\left\vert \alpha \right\vert }=0
\end{equation*}%
has a nonzero solution in $\mathbb{Z}/p\mathbb{Z}$ for infinitely many
primes $p\in \mathbb{Z}$. If $J$ is homogeneous, this implies that%
\begin{equation*}
\sum_{\alpha \in J}\frac{1}{\alpha !}\frac{\partial ^{\alpha }P}{\partial
x^{\alpha }}\left( \boldsymbol{a}\right) =0\text{.}
\end{equation*}
\end{lemma}

\begin{proof}
Fix a prime number $p$. Then since $P$ satisfies the minimal Rado condition,
there exists a root $a\in \mathbb{Z}$ of $\tilde{P}$ and a minimal Rado
functional $\left( J_{0},\ldots ,J_{\ell },d_{1},\ldots ,d_{m}\right) $ for $%
P^{\left( a\right) }$ such that, setting $d_{0}:=0$, the equation
\begin{equation*}
\sum_{i=0}^{m}p^{d_{i}}\sum_{\alpha \in J_{i}}\frac{1}{\alpha !}\frac{%
\partial ^{\alpha }P}{\partial x^{\alpha }}\left( \boldsymbol{a}\right)
w^{\left\vert \alpha \right\vert }=0
\end{equation*}%
has an invertible solution in $\mathbb{Z}_{p}$. Considering the canonical
quotient map $\mathbb{Z}_{p}\rightarrow \mathbb{Z}/p\mathbb{Z}$, we obtain
that the equation%
\begin{equation*}
\sum_{\alpha \in J_{0}}\frac{1}{\alpha !}\frac{\partial ^{\alpha }P}{%
\partial x^{\alpha }}\left( \boldsymbol{a}\right) w^{\left\vert \alpha
\right\vert }=0
\end{equation*}%
has a nonzero solution in $\mathbb{Z}/p\mathbb{Z}$. Notice that, by
definition, $J_{0}$ is a minimal Rado set of $P^{\left( a\right) }$.

As $\tilde{P}$ has finitely many roots, there exists a root $a$ of $\tilde{P}
$ that is obtained as above from infinitely many primes $p$. For this root $%
a $, as $P^{\left( a\right) }$ has finitely many minimal Rado sets, there
exists a minimal Rado set $J_{0}$ of $P^{\left( a\right) }$ that is obtained
as above from infinitely many primes $p$. This concludes the proof of the
first assertion.

In the case when $J$ is homogeneous of degree $d$, we have that%
\begin{equation*}
\sum_{\alpha \in J_{0}}\frac{1}{\alpha !}\frac{\partial ^{\alpha }P}{%
\partial x^{\alpha }}\left( \boldsymbol{a}\right) w^{\left\vert \alpha
\right\vert }=\sum_{\alpha \in J_{0}}\frac{1}{\alpha !}\frac{\partial
^{\alpha }P}{\partial x^{\alpha }}\left( \boldsymbol{a}\right) w^{d}\text{.}
\end{equation*}%
Therefore, the equation%
\begin{equation*}
\sum_{\alpha \in J_{0}}\frac{1}{\alpha !}\frac{\partial ^{\alpha }P}{%
\partial x^{\alpha }}\left( \boldsymbol{a}\right) w^{\left\vert \alpha
\right\vert }=0
\end{equation*}%
has a nonzero solution in $\mathbb{Z}/p\mathbb{Z}$ if and only if%
\begin{equation*}
\sum_{\alpha \in J_{0}}\frac{1}{\alpha !}\frac{\partial ^{\alpha }P}{%
\partial x^{\alpha }}\left( \boldsymbol{a}\right) \equiv 0\mathrm{\mathrm{\
\mathrm{mod}}\ }p\text{.}
\end{equation*}%
If this holds for infinitely many primes $p$, this implies that%
\begin{equation*}
\sum_{\alpha \in J_{0}}\frac{1}{\alpha !}\frac{\partial ^{\alpha }P}{%
\partial x^{\alpha }}\left( \boldsymbol{a}\right) =0\text{.}
\end{equation*}%
Thus, the second assertion follows from the first one.
\end{proof}

In the next section, we will prove the following \emph{necessary }conditions
for partition regularity in terms of the Rado conditions.

\begin{theorem}
\label{Theorem:necessary}Fix $P\in \mathbb{Z}\left[ x_{1},\ldots ,x_{n}%
\right] $. Suppose that $P$ is partition regular. Then:

\begin{enumerate}
\item $P$ satisfies the maximal Rado condition;

\item if $\tilde{P}$ is nonzero and splits over $\mathbb{Z}$ as a product of
linear factors, then $P$ satisfies the minimal Rado condition.
\end{enumerate}
\end{theorem}

The first assertion of Theorem \ref{Theorem:necessary} is proved in Theorem %
\ref{Theorem:maximal} below, while the second assertion of Theorem \ref%
{Theorem:necessary} is an immediate consequence of Proposition \ref%
{Lemma:root} and Theorem \ref{Theorem:minimal}.

We observe that Part (2) of Theorem \ref{Theorem:necessary} recovers, in the
particular case of homogeneous linear polynomial, one direction in the Rado
characterization of partition regular homogeneous linear equations.
Furthermore, Corollary 3.9 of \cite{di_nasso_ramsey_2018} is the particular
instance of Part (2) of Theorem \ref{Theorem:necessary} when $P$ is
homogeneous. In view of Corollary \ref{Corollary:length} below, Theorem 3.10
of \cite{di_nasso_ramsey_2018} is the particular instance of Part (1) of
Theorem \ref{Theorem:necessary} when $P\left( x_{1},\ldots ,x_{n}\right) $
is of the form $P_{1}\left( x_{1}\right) +\cdots +P_{n}\left( x_{n}\right) $
for some monovariate polynomials $P_{1},\ldots ,P_{n}$.

It would be interesting to know whether the Rado conditions provide \emph{%
sufficient }criteria for partition regularity for more general Diophantine
equations, at least for equations $P\left( x_{1},\ldots ,x_{n}\right) =0$
with infinitely many solutions in $\mathbb{N}$ and such that the
corresponding monovariate polynomial $\tilde{P}$ is a homogeneous \emph{%
linear} polynomial. Such a result would be a natural generalization of
Rado's characterization of homogeneous linear equations from \cite%
{rado_studien_1933}.

\begin{conjecture}
\label{Conjecture:crazy}Let $P\in \mathbb{Z}\left[ x_{1},\ldots ,x_{n}\right]
$ be a polynomial such that $\tilde{P}\left( w\right) :=P\left( w,\ldots
,w\right) $ is a nonzero homogeneous linear polynomial. Suppose that the
equation $P\left( x_{1},\ldots ,x_{n}\right) =0$ has infinitely many
solutions in $\mathbb{N}$. Then, the following assertions are equivalent:

\begin{enumerate}
\item $P$ is partition regular;

\item $P$ satisfies the maximal Rado condition and the minimal Rado
condition.
\end{enumerate}
\end{conjecture}

We provide evidence towards Conjecture \ref{Conjecture:crazy} by considering
the families of polynomials $x^{2}-xy+ax+by+cz$ and $x^{2}-y^{2}+ax+by+cz$
for $a,b,c,k\in \mathbb{Z}$. We establish that Conjecture \ref%
{Conjecture:crazy} holds for such families, as long as either the product or
the sum of the coefficients is zero.

\begin{theorem}
\label{Theorem:equation}Fix $a,b,c\in \mathbb{Z}$. Assume that either $abc=0$
or $a+b+c=0$. Suppose that $P\left( x,y,z\right) $ is either the polynomial $%
x^{2}-xy+ax+by+cz$, or the polynomial $x^{2}-y^{2}+ax+by+cz$. Then the
following assertions are equivalent:

\begin{enumerate}
\item $P$ is partition regular;

\item $P$ satisfies the minimal Rado condition whenever $\tilde{P}$ is
nonzero.
\end{enumerate}
\end{theorem}

\begin{remark}
Notice that the maximal Rado condition is always satisfied for a polynomial $%
P$ as in Theorem \ref{Theorem:equation}.
\end{remark}

The only case left out by Theorem \ref{Theorem:equation} is when $a$, $b$, $%
c $, and $a+b+c$ are all nonzero. In this case, $P$ satisfies the minimal
Rado condition if and only if $0\in \left\{ a+b,a+c,b+c\right\} $. Hence,
only the following six cases remain.

\begin{problem}
Fix $a,b\in \mathbb{Z}\setminus \left\{ 0\right\} $. Are the following
polynomials partition regular?

\begin{itemize}
\item $x\left( x-y\right) -xy+a\left( x-y\right) +bz$

\item $x\left( x-y\right) +a\left( y-z\right) +bx$

\item $x\left( x-y\right) +a\left( x-z\right) +by$

\item $\left( x+y\right) \left( x-y\right) +a\left( x-y\right) +bz$

\item $\left( x+y\right) \left( x-y\right) +a\left( y-z\right) +bx$

\item $\left( x+y\right) \left( x-y\right) +a\left( x-z\right) +by$
\end{itemize}
\end{problem}

In order to prove Theorem \ref{Theorem:equation}, we will use the recent
result from \cite{moreira_monochromatic_2017} asserting that the
configuration $\left\{ x+p\left( y\right) ,x+q\left( y\right) ,xy\right\} $
is partition regular whenever $p\left( y\right) ,q\left( y\right) $ are
polynomials with integer coefficients vanishing at $0$ (see Theorem \ref%
{Theorem:Moreira} below). In other words, for every finite coloring of $%
\mathbb{N}$ there exist infinitely many $a,b\in \mathbb{N}$ such that the
set $\left\{ a+p\left( b\right) ,a+q\left( b\right) ,ab\right\} $ is
monochromatic. In fact, we will need a similar result for the configurations
\begin{equation*}
\left\{ x+p\left( y\right) ,x+q\left( y\right) ,xy+x+dy\right\}
\end{equation*}
where $d\in \mathbb{Q}$; see Theorem \ref{Theorem:PR}.

Theorem \ref{Theorem:necessary} will be proved using the formalism of
hypernatural numbers and nonstandard extensions. We will adopt the notation
and terminology from \cite{di_nasso_ramsey_2018}. Particularly, we will
consider a $\mathfrak{c}^{+}$-saturated nonstandard extension ${}^{\ast }%
\mathbb{N}$ of $\mathbb{N}$. Any finite coloring $c:\mathbb{N}\rightarrow
\left\{ 1,2,\ldots ,k \right\} $ extends to a coloring of ${}^{\ast }\mathbb{%
N}$, which we still denote by $c$. If $\xi ,\eta \in {}^{\ast }\mathbb{N}$,
we write $\xi \ll \eta $ and $\eta \gg \xi $ if $\eta -\xi $ is infinite.
Two hypernatural numbers $\xi ,\eta $ are \emph{indiscernible }(or $u$%
-equivalent) if $c\left( \xi \right) =c\left( \eta \right) $ for every
finite coloring $c$ of $\mathbb{N}$; see \cite[Definition 3.1]%
{di_nasso_ramsey_2018}. This is equivalent to the assertion that, for every $%
A\subseteq \mathbb{N}$, $\xi \in {}^{\ast }A$ if and only if $\eta \in
{}^{\ast }A$. To denote that $\xi ,\eta $ are indiscernible, we write $\xi
\sim \eta $. Several properties of such relation are listed in \cite%
{di_nasso_hypernatural_2015,di_nasso_iterated_2015,di_nasso_taste_2015}. In
particular, we have the following:

\begin{itemize}
\item If $\xi ,\eta \in {}^{\ast }\mathbb{N}$ are indiscernible and $\xi
<\eta $, then $\xi \ll \eta $;

\item If $\xi ,\eta \in {}^{\ast }\mathbb{N}$ are indiscernible, and $f:%
\mathbb{N}\rightarrow \mathbb{N}$, then $f\left( \xi \right) ,f\left( \eta
\right) \in {}^{\ast }\mathbb{N}$ are indiscernible;

\item If $\xi \in {}^{\ast }\mathbb{N}$ and $f:\mathbb{N}\rightarrow \mathbb{%
N}$ are such that $\xi \sim f\left( \xi \right) $, then $\xi =f\left( \xi
\right) $;

\item If $\xi \in {}^{\ast }\mathbb{N}$ and $k\in \mathbb{N}$ are such that $%
\xi \sim k$, then $\xi =k$.
\end{itemize}

Fix $\xi \in {}^{\ast }\mathbb{N}$ and $q\in \mathbb{N}$ with $q\geq 2$. By
transfer, $\xi $ admits a unique \emph{base }$q$ \emph{expansion}. This is
an internal sequence $\left( a_{i}\right) _{i\in {}^{\ast }\mathbb{N}_{0}}$
in $\left\{ 0,1,\ldots ,q-1\right\} $ such that%
\begin{equation*}
\mathrm{Supp}_{q}\left( \xi \right) :=\left\{ i\in {}^{\ast }\mathbb{N}%
_{0}:a_{i}\neq 0\right\}
\end{equation*}%
is hyperfinite, and $\sum_{i\in {}^{\ast }\mathbb{N}}a_{i}q^{i}=\xi $. If $%
\sigma $ is the least element of $\mathrm{Supp}_{q}\left( \xi \right) $, and
$\tau $ is the largest element of $\mathrm{Supp}_{q}\left( \xi \right) $,
then we call $\tau $ the position of the first nonzero digit in the base $q$
expansion of $\xi $, and we call $\sigma $ the position of the last nonzero
digit of the base $q$ decomposition of $\xi $. Accordingly, we call $a_{\tau
}$ the first nonzero digit (in the base $q$ expansion of $\xi $), and $%
a_{\sigma }$ the last nonzero digit (in the base $q$ decomposition of $\xi $%
). Similarly, if $i\in \left\{ 1,2,\ldots ,\tau \right\} $, we call $a_{\tau
-i}$ the $i$-th digit, and $a_{\sigma +i}$ the $i$-th to last digit.

It follows from the properties of indiscernible pairs that, when $q\in
\mathbb{N}\setminus \left\{ 1\right\} $ and $\xi ,\xi ^{\prime }\in {}^{\ast
}\mathbb{N}$ are indiscernible, if $\sigma ,\sigma ^{\prime }$ are the
positions of the last nonzero digits of $\xi ,\xi ^{\prime }$, respectively,
then $\sigma ,\sigma ^{\prime }$ are indiscernible. The same applies to the
positions $\tau ,\tau ^{\prime }$ of the first nonzero digits of $\xi ,\xi
^{\prime }$. Furthermore, we have that the last nonzero digits of $\xi ,\xi
^{\prime }$ are equal, and the first nonzero digits of $\xi ,\xi ^{\prime }$
are equal as well. More generally, for every $i\in \mathbb{N}$, the $i$-th
digits of $\xi ,\xi ^{\prime }$ are equal, and the $i$-th to last digits of $%
\xi ,\xi ^{\prime }$ are equal. It follows by overspill that there exists $%
\nu \in {}^{\ast }\mathbb{N}$ infinite such that, for every $i\leq \nu $,
the $i$-th digits of $\xi ,\xi ^{\prime }$ are equal, and the $i$-th to last
digits of $\xi ,\xi ^{\prime }$ are equal. If $\xi \in {}^{\ast }\mathbb{N}$
and $q\in \mathbb{N}\setminus \left\{ 1\right\} $, then we say that $\xi $
has no finite tail in base $q$ if the position of its last nonzero digit in
the base $q$ expansion of $\xi $ is infinite. This is equivalent to the
assertion that $\xi \equiv 0\mathrm{\ \mathrm{mod}}\ q^{n}$ for every $n\in
\mathbb{N}$.

The following characterization of partition regularity is well-known; see
for example \cite[Proposition 3.2]{di_nasso_ramsey_2018}.

\begin{remark}
The equation $P\left( x_{1},\ldots ,x_{n}\right) =0$ is partition regular if
and only if there exist $\xi _{1},\ldots ,\xi _{n}\in {}^{\ast }\mathbb{N}$
which are infinite, indiscernible, and such that $P\left( \xi _{1},\ldots
,\xi _{n}\right) =0$.
\end{remark}

In the same way, one can prove the following characterization of $q$%
-partition regular equations.

\begin{remark}
\label{Remark:p-partitionregular} Fix $q\in \mathbb{N}\setminus \left\{
1\right\} $. The equation $P\left( x_{1},\ldots ,x_{n}\right) =0$ is $q$%
-partition regular if and only if there exist $\xi _{1},\ldots ,\xi _{n}\in
{}^{\ast }\mathbb{N}$ which are infinite, indiscernible, with no finite tail
in base $q$, and such that $P\left( \xi _{1},\ldots ,\xi _{n}\right) =0$.
\end{remark}

One can also characterize upper and lower Rado functionals, as follows.

\begin{remark}
\label{Remark:partition}Let $P\in \mathbb{Z}\left[ x_{1},\ldots ,x_{n}\right]
$ be a polynomial. Then $\left( J_{0},\ldots J_{\ell },d_{0},\ldots
,d_{m-1}\right) $ is an upper Rado functional for $P$ if and only if there
exist infinite, indiscernible $\tau _{1},\ldots ,\tau _{n}\in {}^{\ast }%
\mathbb{N}$ and $M_{0},\ldots ,M_{\ell }\in {}^{\ast }\mathbb{N}$ such that $%
M_{\ell }<M_{\ell -1}<\cdots <M_{0}$,
\begin{equation*}
J_{i}=\left\{ \alpha \in \mathrm{Supp}\left( P\right) :\alpha _{1}\tau
_{1}+\cdots +\alpha _{n}\tau _{n}=M_{i}\right\}
\end{equation*}%
for $i\in \left\{ 0,1,\ldots ,\ell \right\} $,%
\begin{equation*}
d_{i}=M_{i}-M_{m}
\end{equation*}%
for $i\in \left\{ 0,1,\ldots ,m-1\right\} $, and $M_{m}-M_{i}$ is infinite
for $i\in \left\{ m+1,\ldots \ell \right\} $.

Similarly, $\left( J_{0},\ldots J_{\ell },d_{0},\ldots ,d_{m-1}\right) $ is
a lower Rado functional if and only if there exist infinite, indiscernible $%
\sigma _{1},\ldots ,\sigma _{n}\in {}^{\ast }\mathbb{N}$ and $M_{0},\ldots
,M_{\ell }\in {}^{\ast }\mathbb{N}$ such that $M_{0}<M_{1}<\cdots <M_{\ell }$%
,
\begin{equation*}
J_{i}=\left\{ \alpha \in \mathrm{Supp}\left( P\right) :\alpha _{1}\sigma
_{1}+\cdots +\alpha _{n}\sigma _{n}=M_{i}\right\}
\end{equation*}%
for $i\in \left\{ 0,1,\ldots ,\ell \right\} $,%
\begin{equation*}
d_{i}=M_{i}-M_{0}
\end{equation*}%
for $i\in \left\{ 1,\ldots ,m\right\} $, and $M_{i}-M_{m}$ is infinite for $%
i\in \left\{ m+1,\ldots \ell \right\} $.
\end{remark}

\begin{corollary}
\label{Corollary:length}If $\left( J_{0},\ldots ,J_{\ell
},d_{0},d_{1},\ldots ,d_{m}\right) $ is an upper or lower Rado functional, $%
i,j\in \left\{ 0,1,\ldots ,m\right\} $, $\alpha \in J_{i}$ and $\beta \in
J_{j}$ are such that $\ell \left( \alpha +\beta \right) \leq 2$, then $i=j$
and $\left\vert \alpha \right\vert =\left\vert \beta \right\vert $.
\end{corollary}

\begin{proof}
We consider the case when $\left( J_{0},\ldots ,J_{\ell },d_{0},d_{1},\ldots
,d_{m}\right) $ is an upper Rado functional, as the other case is analogous.
Suppose that $\tau _{1},\ldots ,\tau _{n}\in {}^{\ast }\mathbb{N}$ and $%
M_{\ell }<\cdots <M_{0}$ are obtained from $\left( J_{0},\ldots ,J_{\ell
},d_{0},d_{1},\ldots ,d_{m}\right) $ as in the previous remark. Without loss
of generality, we can assume that $i\leq j$, $\alpha =\left( \alpha
_{1},\alpha _{2},0,\ldots ,0\right) $ and $\beta =\left( \beta _{1},\beta
_{2},0,\ldots ,0\right) $. Then we have that%
\begin{equation*}
\alpha _{1}\tau _{1}+\alpha _{2}\tau _{2}=M_{i}
\end{equation*}%
and%
\begin{equation*}
\beta _{1}\tau _{1}+\beta _{2}\tau _{2}=M_{j}\text{.}
\end{equation*}%
Observe that $M_{i}-M_{j}$ is finite. Define $\gamma _{1}:=\alpha _{1}-\beta
_{1}$ and $\gamma _{2}:=\alpha _{2}-\beta _{2}$. Thus,%
\begin{equation*}
M_{i}-M_{j}=\gamma _{1}\tau _{1}+\gamma _{2}\tau _{2}
\end{equation*}%
and%
\begin{equation*}
\gamma _{1}\tau _{1}=-\gamma _{2}\tau _{2}+M_{i}-M_{j}\sim -\gamma _{2}\tau
_{1}+M_{i}-M_{j}\text{.}
\end{equation*}%
It follows that%
\begin{equation*}
\gamma _{1}\tau _{1}=-\gamma _{2}\tau _{1}+M_{i}-M_{j}
\end{equation*}%
and hence%
\begin{equation*}
\left( \gamma _{1}+\gamma _{2}\right) \tau _{1}=M_{i}-M_{j}\text{.}
\end{equation*}%
Since $\tau _{1}$ is infinite, we have that $\gamma _{1}=-\gamma _{2}$, $%
M_{i}=M_{j}$, and $i=j$. This concludes the proof.
\end{proof}

\section{Necessity of the Rado conditions\label{Section:necessity}}

In this section, we assume that $P\left( x_{1},\ldots ,x_{n}\right)
=\sum_{\alpha }c_{\alpha }x^{\alpha }\in \mathbb{Z}\left[ x_{1},\ldots ,x_{n}%
\right] $ is a polynomial with integer coefficients.

\begin{theorem}
\label{Theorem:maximal}Suppose that $P$ is partition regular. Then $P$
satisfies the maximal Rado condition.
\end{theorem}

\begin{proof}
For $\xi ,\eta \in {}^{\ast }\mathbb{R}$ finite, we write $\xi \thickapprox
\eta $ if $\xi -\eta $ is infinitesimal. We also let $\mathrm{st}\left( \xi
\right) $ be the standard part of $\xi $. Fix $q\in \mathbb{N}\setminus
\left\{ 1\right\} $. Set $c_{\alpha }:=\frac{1}{\alpha !}\frac{\partial
^{\alpha }P}{\partial x^{\alpha }}\left( \boldsymbol{0}\right) $ for every
index $\alpha $. By assumption there exist $\xi _{1},\ldots ,\xi _{n}\in
{}^{\ast }\mathbb{N}$ such that $\xi _{1},\ldots ,\xi _{n}$ are
indiscernible, infinite, and $P\left( \xi _{1},\ldots ,\xi _{n}\right) =0$.
For $i\in \left\{ 1,2,\ldots ,n\right\} $, let $\tau _{i}$ be the position
of the first nonzero digit in the base $q$ expansion of $\xi _{i}$. Observe
that $\tau _{1},\ldots ,\tau _{n}$ are infinite and indiscernible. Set $\tau
:=\left( \tau _{1},\ldots ,\tau _{n}\right) $. For an index $\alpha $, we
define $\alpha \cdot \tau $ to be $\alpha _{1}\tau _{1}+\cdots +\alpha
_{n}\tau _{n}$. Let $\left( M_{\ell },\ldots ,M_{0}\right) $ be the
increasing enumeration of $\left\{ \alpha \cdot \tau :\alpha \in \mathrm{%
\mathrm{Supp}}\left( P\right) \right\} $. Define $m$ to be the least element
of the set $\left\{ t\in \left\{ 0,1,\ldots ,\ell \right\} :M_{t+1}\ll
M_{0}\right\} $. For $t\in \left\{ 0,1,\ldots ,\ell \right\} $, define%
\begin{equation*}
J_{t}=\left\{ \alpha \in \mathrm{Supp}\left( P\right) :\alpha \cdot \tau
=M_{t}\right\} \text{.}
\end{equation*}%
For $i\in \left\{ 0,1,\ldots ,m\right\} $, define $d_{i}:=M_{i}-M_{m}$.
Observe that $\left( J_{0},\ldots ,J_{\ell },d_{0},\ldots ,d_{m}\right) $ is
an upper Rado functional for $P$, as witnessed by $\left( \tau _{1},\ldots
,\tau _{n}\right) $; see Remark \ref{Remark:partition}.

Fix now an infinite $\nu \in {}^{\ast }\mathbb{N}$ such that, for $i\leq \nu
$, the $i$-th digits of $\xi _{1},\ldots ,\xi _{n}$ in the base $q$
expansion are equal. We can write%
\begin{equation*}
\xi _{i}=q^{\tau _{i}}\left( \sigma +q^{-\nu }\rho _{i}\right)
\end{equation*}%
where $1\leq \sigma <q$ and $\rho _{1},\ldots ,\rho _{n}\in {}^{\ast }%
\mathbb{R}$ are such that $0\leq \rho _{i}\leq 1$ for every $i\in \left\{
1,2,\ldots ,n\right\} $. Define $\zeta _{i}:=\sigma +q^{-\nu }\rho _{i}$ for
$i\in \left\{ 1,2,\ldots ,n\right\} $ and, for $t\in \left\{ 0,1,\ldots
,\ell \right\} $,%
\begin{equation*}
Q_{t}\left( x\right) :=\sum_{\alpha \in J_{t}}c_{\alpha }x^{\alpha }\text{.}
\end{equation*}%
Set also $w=\mathrm{st}\left( \sigma \right) \in \mathbb{R}$ and observe
that $1\leq w\leq q$. Then we have that%
\begin{equation*}
0=P\left( \xi \right) =\sum_{t=0}^{\ell }Q_{t}\left( \xi \right)
=\sum_{t=0}^{\ell }q^{M_{t}}Q_{t}\left( \zeta _{1},\ldots ,\zeta _{n}\right)
\text{.}
\end{equation*}%
Thus%
\begin{equation*}
\sum_{t=0}^{m}q^{d_{t}}Q_{t}\left( \zeta _{1},\ldots ,\zeta _{n}\right)
=-\sum_{t=m+1}^{\ell }\frac{1}{q^{M_{m}-M_{t}}}Q_{t}\left( \zeta _{1},\ldots
,\zeta _{n}\right) \thickapprox 0
\end{equation*}%
is infinitesimal. Set $\rho :=\left( \rho _{1},\ldots ,\rho _{n}\right) $.
For $t\in \left\{ \ell -m,...,\ell \right\} $, we have%
\begin{eqnarray*}
Q_{t}\left( \zeta _{1},\ldots ,\zeta _{n}\right) &=&Q_{t}\left( \sigma
+q^{-\nu }\rho _{1},\ldots ,\sigma +q^{-\nu }\rho _{n}\right) \\
&=&\sum_{\alpha \in J_{t}}\sum_{\beta \leq \alpha }c_{\alpha }\sigma
^{\left\vert \alpha \right\vert -\left\vert \beta \right\vert }q^{-\nu
\left\vert \beta \right\vert }\rho ^{\beta } \\
&=&\sum_{\alpha \in J_{t}}c_{\alpha }\sigma ^{\left\vert \alpha \right\vert
}+\sum_{\alpha \in J_{t}}\sum_{\substack{ \beta \leq \alpha  \\ \beta \neq
\boldsymbol{0}}}c_{\alpha }\sigma ^{\left\vert \alpha \right\vert
-\left\vert \beta \right\vert }q^{-\nu \left\vert \beta \right\vert }\rho
^{\beta }\text{.}
\end{eqnarray*}%
Observe that%
\begin{equation*}
\sum_{\alpha \in J_{t}}\sum_{\substack{ \beta \leq \alpha  \\ \beta \neq
\boldsymbol{0}}}c_{\alpha }\sigma ^{\left\vert \alpha \right\vert
-\left\vert \beta \right\vert }q^{-\nu \left\vert \beta \right\vert }\rho
^{\beta }\thickapprox 0
\end{equation*}%
is infinitesimal. Therefore, we have that%
\begin{equation*}
Q_{t}\left( \zeta _{1},\ldots ,\zeta _{n}\right) \thickapprox \sum_{\alpha
\in J_{t}}c_{\alpha }\sigma ^{\left\vert \alpha \right\vert }\text{.}
\end{equation*}%
Thus%
\begin{equation*}
0\thickapprox \sum_{t=0}^{m}q^{d_{t}}Q_{t}\left( \zeta _{1},\ldots ,\zeta
_{n}\right) \thickapprox \sum_{t=0}^{m}q^{d_{t}}\sum_{\alpha \in
J_{t}}c_{\alpha }\sigma ^{\left\vert \alpha \right\vert }\text{.}
\end{equation*}%
Considering the standard part, we have%
\begin{equation*}
0=\sum_{t=0}^{m}q^{d_{t}}\sum_{\alpha \in J_{t}}c_{\alpha }w^{\left\vert
\alpha \right\vert }\text{.}
\end{equation*}%
Since $w\in \left[ 1,q\right] $, this concludes the proof.
\end{proof}

\begin{proposition}
\label{Lemma:root}Suppose that the monovariate polynomial $\tilde{P}$
associated with $P$ is nonzero and splits over $\mathbb{Z}$ as a product of
linear factors. If $P$ is partition regular, then for every prime $p\in
\mathbb{N}$ there exists a root $a$ of $\tilde{P}$ such that $P^{\left(
a\right) }$ is $p$-partition regular.
\end{proposition}

\begin{proof}
Define $\ell $ to be the degree of $\tilde{P}$, and suppose that $%
a_{1},\ldots ,a_{\ell }\in \mathbb{Z}$ are the roots of $\tilde{P}$. Suppose
that $P$ is partition regular. Then there exist infinite, indiscernible $\xi
_{1},\ldots ,\xi _{n}\in {}^{\ast }\mathbb{N}$ such that $P\left( \xi
_{1},\ldots ,\xi _{n}\right) =0$. Fix an infinite $\nu \in {}^{\ast }\mathbb{%
N}$ and $\rho \in \left\{ 0,1,\ldots ,p^{\ell \nu }-1\right\} $ such that
\begin{equation*}
\xi _{1}\equiv \cdots \equiv \xi _{n}\equiv \rho \mathrm{\ \mathrm{mod}}\
p^{\ell \nu }\text{,}
\end{equation*}%
where $\left\{ 0,1,\ldots ,p^{\ell \nu }-1\right\} $ denotes the set of $%
x\in {}^{\ast }\mathbb{N}_{0}$ such that $x<p^{\ell \nu }$. Thus, we have
that $\tilde{P}\left( \rho \right) \equiv 0\mathrm{\ \mathrm{mod}}\ p^{\ell
\nu }$, which implies that $p^{\ell \nu }$ divides $\left( \rho
-a_{1}\right) \cdots \left( \rho -a_{\ell }\right) $. Hence, there exists a
root $a$ of $\tilde{P}$ such that $p^{\nu }$ divides $\rho -a$. This implies
that
\begin{equation*}
\xi _{1}-a\equiv \cdots \equiv \xi _{n}-a\equiv \rho -a\equiv 0\mathrm{\
\mathrm{mod}}\ p^{\nu }\text{.}
\end{equation*}%
Define now $\eta _{i}:=\xi _{i}-a$ for $i\in \left\{ 1,2,\ldots ,n\right\} $%
. Observe that $\eta _{1},\ldots ,\eta _{n}$ are indiscernible, infinite,
with no finite tail in base $p$, and such that $P^{\left( a\right) }\left(
\eta _{1},\ldots ,\eta _{n}\right) =0$. The conclusion thus follows from
Remark \ref{Remark:p-partitionregular}.
\end{proof}

\begin{theorem}
\label{Theorem:minimal}Suppose that $p\in \mathbb{N}$ is a prime. If $P$ is $%
p$-partition regular, then there exists a minimal Rado functional $\left(
J_{0},\ldots ,J_{\ell },d_{1},\ldots ,d_{m}\right) $ for $P$ such that,
setting $d_{0}=0$, the equation
\begin{equation*}
\sum_{i=0}^{m}p^{d_{i}}\sum_{\alpha \in J_{i}}\frac{1}{\alpha !}\frac{%
\partial ^{\alpha }P}{\partial x^{\alpha }}\left( \boldsymbol{0}\right)
w^{\left\vert \alpha \right\vert }= 0
\end{equation*}%
has an invertible solution in the ring $\mathbb{Z}_{p}$ of $p$-adic integers.
\end{theorem}

\begin{proof}
Suppose that $P$ is $p$-partition regular. Thus, there exist infinite,
indiscernible $\xi _{1},\ldots ,\xi _{n}\in {}^{\ast }\mathbb{N}$ with no
finite tail in base $p$ such that $P\left( \xi _{1},\ldots ,\xi _{n}\right)
=0$. For $i\in \left\{ 1,2,\ldots ,n\right\} $, let $\sigma _{i}$ be the
position of the last nonzero digit in the base $p$ expansion of $\xi _{i}$.
Observe that $\sigma _{1},\ldots ,\sigma _{n}$ are infinite and
indiscernible. Set $\sigma :=\left( \sigma _{1},\ldots ,\sigma _{n}\right) $%
. For an index $\alpha $, define $\alpha \cdot \sigma $ to be $\alpha
_{1}\sigma _{1}+\cdots +\alpha _{n}\sigma _{n}$. Let $\left( M_{0},\ldots
,M_{\ell }\right) $ be the increasing enumeration of $\left\{ \alpha \cdot
\sigma :\alpha \in \mathrm{\mathrm{Supp}}\left( P\right) \right\} $. Define $%
m$ be the least element of the set $\left\{ i\in \left\{ 0,1,\ldots ,\ell
\right\} :M_{0}\ll M_{i+1}\right\} $. For $t\in \left\{ 0,1,\ldots ,\ell
\right\} $, define%
\begin{equation*}
J_{t}=\left\{ \alpha \in \mathrm{Supp}\left( P\right) :\alpha \cdot \sigma
=M_{t}\right\} \text{.}
\end{equation*}%
For $i\in \left\{ 0,1,\ldots ,m\right\} $, define $d_{i}=M_{i}-M_{0}$.
Observe that $\left( J_{0},\ldots ,J_{\ell },d_{0},\ldots ,d_{m}\right) $ is
a lower Rado functional for $P$, as witnessed by $\left( \sigma _{1},\ldots
,\sigma _{n}\right) $. For $i\in \left\{ 1,2,\ldots ,n\right\} $, let $\zeta
_{i}\in {}^{\ast }\mathbb{N}$ be such that%
\begin{equation*}
\xi _{i}=p^{\sigma _{i}}\zeta _{i}\text{.}
\end{equation*}%
Define $\zeta =\left( \zeta _{1},\ldots ,\zeta _{n}\right) $. Observe that $%
\zeta _{1},\ldots ,\zeta _{n}$ are indiscernible and, for $\alpha \in J_{t}$%
, we have that%
\begin{equation*}
\xi ^{\alpha }=p^{M_{t}}\zeta ^{\alpha }\text{.}
\end{equation*}%
Fix $\nu \in {}^{\ast }\mathbb{N}$ infinite such that $\zeta _{1}\equiv
\cdots \equiv \zeta _{n}\mathrm{\ \mathrm{mod}}\ p^{\nu }$ and $M_{m}+\nu
\ll M_{m+1}$. Let $\rho \in \left\{ 1,\ldots ,p^{\nu }-1\right\} $ be such
that $\zeta _{1}\equiv \rho \mathrm{\ \mathrm{mod}}\ p^{\nu }$, where as
before $\left\{ 1,\ldots ,p^{\nu }-1\right\} $ denotes the set of $x\in
{}^{\ast }\mathbb{N}$ such that $x<p^{\nu }$. Observe that, by definition of $%
\zeta _{1}$ and $\rho $, $p$ does not divide $\rho $. To ease the notation,
let $c_{\alpha }:=\frac{1}{\alpha !}\frac{\partial ^{\alpha }P}{\partial
x^{\alpha }}\left( \boldsymbol{0}\right) $. We have that%
\begin{equation*}
0=\sum_{\alpha }c_{\alpha }\xi ^{\alpha }\equiv
\sum_{i=0}^{m}p^{M_{i}}\sum_{\alpha \in J_{i}}c_{\alpha }\zeta ^{\alpha
}=p^{M_{0}}\sum_{i=0}^{m}p^{M_{i}-M_{0}}\sum_{\alpha \in J_{i}}c_{\alpha
}\zeta ^{\alpha }\mathrm{\ \mathrm{mod}}\ p^{M_{m}+\nu }\text{.}
\end{equation*}%
Hence,%
\begin{equation*}
0\equiv \sum_{i=0}^{m}p^{M_{i}-M_{0}}\sum_{\alpha \in J_{t}}c_{\alpha }\zeta
^{\alpha }\equiv \sum_{i=0}^{m}p^{d_{i}}\sum_{\alpha \in J_{t}}c_{\alpha
}\rho ^{|\alpha |}\mathrm{\ \mathrm{mod}}\ p^{\nu }\text{.}
\end{equation*}%
In particular, for every $k\in \mathbb{N}$,%
\begin{equation*}
\sum_{i=0}^{m}p^{d_{i}}\sum_{\alpha \in J_{i}}c_{\alpha }\rho ^{|\alpha
|}\equiv 0\mathrm{\ \mathrm{mod}}\ p^{k}\text{.}
\end{equation*}%
Since $p$ does not divide $\rho $, this witnesses that the equation $%
\sum_{i=0}^{m}p^{d_{i}}\sum_{\alpha \in J_{i}}c_{\alpha }w^{\left\vert
\alpha \right\vert }\equiv 0\mathrm{\ \mathrm{mod}}\ p^{k}$ has an
invertible solution in $\mathbb{Z}/p^{k}\mathbb{Z}$. Since this holds for
every $k\in \mathbb{N}$, we have that the equation%
\begin{equation*}
\sum_{i=0}^{m}p^{d_{i}}\sum_{\alpha \in J_{i}}c_{\alpha }w^{|\alpha |}=0
\end{equation*}%
has an invertible solution in $\mathbb{Z}_{p}$.
\end{proof}

\begin{example}
Consider the quadratic polynomial%
\begin{equation*}
P\left( x,y\right) =x^{2}+y^{2}-xy-ax-by+ab
\end{equation*}%
with $a,b\in \mathbb{Z}$. We claim that the equation $P\left( x,y\right) =0$
has only finitely many integer solutions when $a=b$, and it does not satisfy
the minimal Rado condition when $a\neq b$. In particular, $P$ is not
partition regular.

Suppose that $a=b$, in which case%
\begin{equation*}
P\left( x,y\right) =x^{2}+y^{2}-xy-ax-ay+a^{2}\text{.}
\end{equation*}%
Set%
\begin{equation*}
r=x-y-a
\end{equation*}%
and observe that%
\begin{eqnarray*}
P\left( x,y\right) &=&x^{2}+y^{2}-xy-ax-ay+a^{2} \\
&=&r^{2}+xy+ax-3ay \\
&=&y^{2}+y\left( r-a\right) +a^{2}+ar+r^{2}\text{.}
\end{eqnarray*}%
If $x,y$ is an integer solution to $P\left( x,y\right) =0$, we have
necessarily that $\left( r-a\right) ^{2}\geq 4\left( a^{2}+ar+r^{2}\right) $
and hence $a+r=0$. Therefore $x=y$ and%
\begin{equation*}
P\left( x,y\right) =\left( x-a\right) ^{2}\text{.}
\end{equation*}%
Thus the only integer solution to $P\left( x,y\right) =0$ is $\left(
a,a\right) $ when $a=b$.

Consider the case when $a\neq b$. In this case we have
\begin{equation*}
\tilde{P}\left( w\right) =\left( w-a\right) \left( w-b\right) \text{.}
\end{equation*}%
Furthermore%
\begin{eqnarray*}
&&P^{\left( a\right) }\left( x,y\right) \\
&=&\left( x+a\right) ^{2}+\left( y+a\right) ^{2}-\left( x+a\right) \left(
y+a\right) -a\left( x+a\right) -b\left( y+a\right) +ab \\
&=&x^{2}+y^{2}-xy+2ax+2ay-ax-ay-ax-by \\
&=&x^{2}+y^{2}-xy+\left( a-b\right) y\text{.}
\end{eqnarray*}%
Similarly, we have that%
\begin{equation*}
P^{\left( b\right) }\left( x,y\right) =x^{2}+y^{2}-xy+\left( b-a\right) x%
\text{.}
\end{equation*}%
As $a\neq b$, we have that $\left( a-b\right) y$ and $x^{2}$ are the only
minimal terms in $P^{\left( a\right) }$. Therefore, all minimal Rado sets
for $P^{\left( a\right) }$ are singletons. Similarly, all minimal Rado sets
for $P^{\left( b\right) }$ are singletons. By Lemma \ref{Lemma:minimal-Rado}%
, this shows that $P$ does not satisfy the minimal Rado condition when $%
a\neq b$.
\end{example}

\begin{example}
Consider the polynomial%
\begin{equation*}
P\left( x,y\right) =x^{2}-y^{2}+xy-ax-by+ab
\end{equation*}%
with $a,b\in \mathbb{Z}$. We claim that $P$ does not satisfy the minimal
Rado condition when $a\neq b$. In particular, $P$ is not partition regular.

We have that%
\begin{equation*}
\tilde{P}\left( w\right) =\left( w-a\right) \left( w-b\right) .
\end{equation*}%
Furthermore,%
\begin{eqnarray*}
&&P^{\left( a\right) }\left( x,y\right) \\
&=&\left( x+a\right) ^{2}-\left( y+a\right) ^{2}+\left( x+a\right) \left(
y+a\right) -a\left( x+a\right) -b\left( y+a\right) +ab \\
&=&x^{2}-y^{2}+xy+2ax-\left( a+b\right) y
\end{eqnarray*}%
and%
\begin{eqnarray*}
&&P^{\left( b\right) }\left( x,y\right) \\
&=&\left( x+b\right) ^{2}-\left( y+b\right) ^{2}+\left( x+b\right) \left(
y+b\right) -a\left( x+b\right) -b\left( y+b\right) +ab \\
&=&x^{2}-y^{2}+xy+(3b-a)x-2by\text{.}
\end{eqnarray*}%
Since either $a,b$ is nonzero, at least one between $2a$ and $a+b$ is
nonzero. Hence the minimal terms $P^{\left( a\right) }\left( x,y\right) $
are contained in the set $\left\{ 2ax,-\left( a+b\right) y\right\} $. Hence,
all the minimal Rado sets for $P^{\left( a\right) }$ are homogeneous.
Similarly, the terms of minimal indices for $P^{\left( b\right) }$ are
contained in $\left\{ \left( 3b-a\right) x,\left( -2b\right) y\right\} $,
and all the minimal Rado sets for $P^{\left( b\right) }$ are homogeneous.
Since we have%
\begin{eqnarray*}
2a-\left( a+b\right) &=&a-b\neq 0 \\
\left( 3b-a\right) -2b &=&b-a\neq 0
\end{eqnarray*}%
we conclude by Lemma \ref{Lemma:minimal-Rado} that $P$ does not satisfy the
minimal Rado condition.
\end{example}

\section{Partition regularity of polynomial configurations\label%
{Section:sufficient}}

In this section we show, building on results from \cite%
{moreira_monochromatic_2017}, that certain polynomial configurations in $%
\mathbb{N}$ are partition regular. Let us define the following notation:

\begin{notation}
\label{Notation:set}Let $\mathbb{Q}_{+}$ be the set of nonnegative rational
numbers. Given $d\in \mathbb{Q}$ and a subset $E\subseteq \mathbb{Q}_{+}$,
we define%
\begin{equation*}
E+d=\{x\in \mathbb{Q}_{+}\mid x-d\in E\}
\end{equation*}%
and%
\begin{equation*}
E/d=\{x\in \mathbb{Q}_{+}\mid \left\vert dx\right\vert \in E\}\text{.}
\end{equation*}
\end{notation}

\begin{definition}
\label{def:ps}Fix $d\in \mathbb{Q}$. A subset $A\subseteq \mathbb{N}+d$ is
\emph{piecewise syndetic} if there exists some fixed $k\in \mathbb{N}$ such
that $A$ contains arbitrarily long increasing sequences $a_{1}<\cdots <a_{n}$
satisfying $a_{i+1}-a_{i}\leq k$, for all $i\in \{1,\ldots ,n-1\}$.
\end{definition}

\begin{definition}
A subset $A$ of $\mathbb{N}$ is:

\begin{itemize}
\item \emph{IP-set }if there exists an infinite increasing sequence $\left(
a_{k}\right) _{k\in \mathbb{N}}$ such that, for every increasing sequence of
indices $k_{1}<k_{2}<\cdots <k_{n}$ in $\mathbb{N}$, $a_{k_{1}}+a_{k_{2}}+%
\cdots +a_{k_{n}}\in A$;

\item an \emph{IP}$^{\ast }$\emph{-set }if, for every IP-set $B$, $A\cap
B\neq \varnothing $.
\end{itemize}
\end{definition}

We will make use of the following well-known property of piecewise syndetic
subsets of $\mathbb{N}$: if $E\subseteq \mathbb{N}$ is piecewise syndetic
and $c$ is a finite coloring of $\mathbb{N}$, then there exists a $c$%
-monochromatic piecewise syndetic subset of $E$ (see, for instance, \cite[%
Theorem 1.24]{Furstenberg81}).

A polynomial generalization of the classical van der Waerden theorem on
arithmetic progressions was established by Bergelson and Leibman \cite%
{bergelson_polynomial_1996}. We will use the following reformulation; see
for example \cite[Theorem 4.5]{bergelson_partition_2001}.

\begin{theorem}[Bergelson and Leibman]
\label{Theorem:vdw} Fix $d\in \mathbb{Q}$. Let $E\subseteq \mathbb{N}+d$ be
piecewise syndetic, and let $F\subseteq \mathbb{Z}[x]$ be a finite set of
integer-valued polynomials such that $f(0)=0$ for every $f\in F$. Then the
set
\begin{equation*}
\left\{ n\in \mathbb{N}:E\cap \bigcap_{f\in F}(E-f(n))\text{ is piecewise
syndetic}\right\}
\end{equation*}%
is IP*.
\end{theorem}

By replacing the function $n\mapsto f\left( n\right) $ with the function $%
n\mapsto f\left( n+1\right) $, we immediately obtain the following
consequence.

\begin{corollary}
\label{Corollary:vdw} Fix $d\in \mathbb{Q}$. Let $E\subseteq \mathbb{N}+d$
be piecewise syndetic, and let $F\subseteq \mathbb{Z}[x]$ be a finite set of
integer-valued polynomials such that $f(1)=0$ for every $f\in F$. Then the
set
\begin{equation*}
\left\{ n\in \mathbb{N}:E\cap \bigcap_{f\in F}(E-f(n))\text{ is piecewise
syndetic}\right\} -1
\end{equation*}%
is IP*.
\end{corollary}

Using Theorem \ref{Theorem:vdw}, the following partition regularity result
was established in \cite{moreira_monochromatic_2017}.

\begin{theorem}
\label{Theorem:Moreira}Let $c$ be a finite coloring of $\mathbb{N}$. Let $%
s\in \mathbb{N}$ and, for each $i\in \{1,\ldots ,s\}$, let $F_{i}$ be a
finite set of functions $\mathbb{N}^{i}\rightarrow \mathbb{Z}$ such that for
all $f\in F_{i}$ and any $x_{1},\ldots ,x_{i-1}\in \mathbb{N}$, the function
$x\mapsto f(x_{1},\ldots ,x_{i-1},x)$ is a polynomial with integer
coefficients vanishing at $0$. Then, there exist infinitely many $%
x_{0},\ldots ,x_{s}\in \mathbb{N}$ such that the set%
\begin{equation*}
\{x_{0}\cdots x_{j}:0\leq j\leq s\}\cup \{x_{0}\cdots x_{j}+f(x_{j+1},\ldots
,x_{i})\mid 0\leq j<i\leq s,f\in F_{i-j}\}
\end{equation*}%
is monochromatic for $c$.
\end{theorem}

The following corollary is an immediate consequence of Theorem \ref%
{Theorem:Moreira} and saturation.

\begin{corollary}
\label{Corollary:Moreira}Fix $s\in \mathbb{N}$. There exist infinite $\xi
_{0},\ldots ,\xi _{s}\in {}^{\ast }\mathbb{N}$ such that, for every $0\leq
j<i\leq s$, and for every function $f:\mathbb{N}^{i-j}\rightarrow \mathbb{Z}$
such that, for every $x_{j+1},\ldots ,x_{i-1}\in \mathbb{N}$, the function $%
x\mapsto f\left( x_{j+1},\ldots ,x_{i-1},x\right) $ is a polynomial with
integer coefficients vanishing at $0 $, one has that
\begin{equation*}
\xi _{0}\sim \xi _{0}\cdots \xi _{i}\sim \xi _{0}\cdots \xi _{j}+f\left( \xi
_{j+1},\ldots ,\xi _{i}\right) \text{.}
\end{equation*}
\end{corollary}

\begin{remark}
If $\xi _{0},\ldots ,\xi _{s}\in {}^{\ast }\mathbb{N}$ are as in Corollary %
\ref{Corollary:Moreira}, then we have that $\xi _{0}\equiv \cdots \equiv \xi
_{s}\equiv 0\mathrm{\ \mathrm{mod}}\ q$ for every $q\in \mathbb{N}$. Indeed,
we have that, for $0<i\leq s$,%
\begin{equation*}
\xi _{0}\cdots \xi _{i-1}+\xi _{i}\equiv \xi _{0}\cdots \xi _{i-1}\mathrm{\
\mathrm{mod}}\ q
\end{equation*}%
and hence $\xi _{i}\equiv 0\mathrm{\ \mathrm{mod}}\ q$. It follows that $\xi
_{0}\equiv \xi _{0}\xi _{1}\equiv 0\mathrm{\ \mathrm{mod}}\ q$.
\end{remark}

The following result is the analogue of Theorem \ref{Theorem:Moreira} for
polynomials vanishing at $1$. While the proof of Theorem \ref{Theorem:PR}
bears some resemblance to the proof of Theorem \ref{Theorem:Moreira}, some
modifications are needed to deal with polynomials vanishing at $1$. It is
unclear how to obtain an analogue of such results for polynomials vanishing
at a point other than $0$ or $1$. Recall that we let $\mathbb{N}$ be the set
of integer that are greater than or equal to $1$, and $\mathbb{N}_{0}=%
\mathbb{N}\cup \left\{ 0\right\} $. Notice that, according to Notation \ref%
{Notation:set}, $\mathbb{N}\cap (\mathbb{N}_{0}/d+1)=\mathbb{N}$ when $d=0$.

\begin{theorem}
\label{Theorem:PR}Fix $d\in \mathbb{Q}$. Let $c$ be a finite coloring of $%
\mathbb{N}+d$. Let $s\in \mathbb{N}$ and, for each $i\in \{1,\ldots ,s\}$,
let $F_{i}$ be a finite set of functions $\mathbb{N}^{i}\rightarrow \mathbb{Z%
}$ such that for all $f\in F_{i}$ and any $x_{1},\ldots ,x_{i-1}\in \mathbb{N%
}$, the function $x\mapsto f(x_{1},\ldots ,x_{i-1},x)$ is a polynomial with
integer coefficients vanishing at $1$. Then there exist infinitely many $%
x_{0}\in \mathbb{N}+d$ and $x_{1},\ldots ,x_{s}\in \mathbb{N}\cap (\mathbb{N}%
_{0}/d+1)$ such that the set
\begin{equation*}
\{x_{0}\cdots x_{j}:0\leq j\leq s\}\cup \{x_{0}\cdots x_{j}+f(x_{j+1},\ldots
,x_{i})\mid 0\leq j<i\leq s,f\in F_{i-j}\}
\end{equation*}%
is contained in $\mathbb{N}+d$ and monochromatic for $c$.
\end{theorem}

\begin{proof}
Observe that we have that%
\begin{equation*}
\{x_{0}\cdots x_{j}:0\leq j\leq s\}
\end{equation*}%
is contained in $\mathbb{N}+d$, as long as $x_{0}\in \mathbb{N}+d$ and $%
x_{1},\ldots ,x_{s}\in \mathbb{N}\cap (\mathbb{N}_{0}/d+1)$.

Fix a finite coloring $c:\mathbb{N}+d\rightarrow \left\{ 1,2,\ldots
,r\right\} $, and set $C_{i}=\left\{ n\in \mathbb{N}+d:c(n) =i\right\} $ for
$i\in \left\{ 1,2,\ldots ,r\right\} $. We will recursively construct four
sequences:

\begin{itemize}
\item A sequence $(y_{i})_{i=1}^{\infty }\ $in $\mathbb{N}\cap (\mathbb{N}%
_{0}/d+1)$;

\item Two sequences $(B_{i})_{i=0}^{\infty }$ and $(D_{i})_{i=1}^{\infty }$
of piecewise syndetic subsets of $\mathbb{N}+d$;

\item A sequence $(t_{i})_{i=0}^{\infty }$ of elements of $\{1,\ldots ,r\}$;
\end{itemize}

such that $B_{i}\subseteq C_{t_{i}}$ for every $i\geq 0$.\newline

We begin by choosing $t_{0}\in \{1,\ldots ,r\}$ such that $C_{t_{0}}\cap
\left( \mathbb{N}+d\right) $ is piecewise syndetic (which is always possible
since $\mathbb{N}+d$ itself is piecewise syndetic), and let $%
B_{0}:=C_{t_{0}} $. Now, assume that $n\geq 1$ and $%
(y_{j})_{j=1}^{n-1},(B_{j})_{j=0}^{n-1},(D_{j})_{j=1}^{n-1}$ and $%
(t_{j})_{j=0}^{n-1}$ have already been defined. For $i,j\in \mathbb{N}$ with
$i\leq j$, we set $y_{i,j}:=y_{i}y_{i+1}\cdots y_{j}$, whenever these have
been defined.

For each $k\in \{1,\ldots ,s\}$ and every $f\in F_{k}$, we define the
collection $G_{n}(f)$ of all functions $g:\mathbb{Z}\rightarrow \mathbb{Z}$
of the form
\begin{equation}
g(z)=y_{m_{1}+1,n-1}(f(y_{m_{1}+1,m_{2}},\ y_{m_{2}+1,m_{3}},\ \ldots ,\
y_{m_{k}+1,n-1}z)-f(y_{m_{1}+1,m_{2}},\ y_{m_{2}+1,m_{3}},\ \ldots ,\
y_{m_{k}+1,n-1}))\text{\label{Equation:g}}
\end{equation}%
for any possible choice $0\leq m_{1}<m_{2}<\cdots <m_{k}<n$. In %
\eqref{Equation:g}, we adopt the convention that an empty product (such as $%
y_{m_{k}+1}\cdots y_{n-1}$ when $m_{k+1}=n-1$) is equal to $1$. If $k>n$,
then we set $G_{n}(f)$ to be empty. Observe that each $g\in G_{n}(f)$ is a
polynomial with integer coefficients vanishing at $1$. Notice that $\mathbb{N%
}\cap \mathbb{N}_{0}/d$ is an IP-set. By Corollary \ref{Corollary:vdw}, we
can find some $y_{n}\in \mathbb{N}\cap (\mathbb{N}_{0}/d+1)$ such that the
intersection
\begin{equation}
D_{n}:=B_{n-1}\cap \bigcap_{k=1}^{s}\bigcap_{f\in F_{k}}\bigcap_{g\in
G_{n}\left( f\right) }(B_{n-1}-g(y_{n}))
\end{equation}%
is piecewise syndetic. Therefore, $y_{n}D_{n}\subseteq \mathbb{N}+d$ is also
piecewise syndetic. (The fact that $y_{n}\in \mathbb{N}\cap (\mathbb{N}%
_{0}/d+1)$ and $D_{n}\subseteq \mathbb{N}+d$ guarantee that $%
y_{n}D_{n}\subseteq \mathbb{N}+d$.) We then define $B_{n}$ to be a $c$%
-monochromatic piecewise syndetic subset of $y_{n}D_{n}$, and let $t_{n}\in
\left\{ 1,2,\ldots ,r\right\} $ be such that $B_{n}\subseteq C_{t_{n}}$.
This concludes the recursive construction.

We note that, for every $i\in \mathbb{N}$, $B_{i}\subseteq
y_{i}D_{i}\subseteq y_{i}B_{i-1}$. This implies that, for every $j<i$, $%
B_{i}\subseteq y_{i}y_{i-1}\cdots y_{j+1}B_{j}$. Since $(t_{i})_{i=0}^{%
\infty }$ takes only finitely many values, there exists $t\in \{1,\ldots
,r\} $ and infinitely many tuples of natural numbers $n_{0}<\cdots <n_{s}$
such that $t_{n_{0}}=t_{n_{1}}=\cdots =t_{n_{s}}=t$. Consider the elements $%
x_{1}:=y_{n_{0}+1,n_{1}}$, $x_{2}:=y_{n_{1}+1,n_{2}}$, $\ldots $, $%
x_{s}:=y_{n_{s-1}+1,n_{s}}$ of $\mathbb{N}+1$. Pick then $\tilde{x}\in
B_{n_{s}}$, and observe that%
\begin{equation*}
\tilde{x}\in B_{n_{s}}\subseteq x_{s}B_{n_{s-1}}\subseteq
x_{s}x_{s-1}B_{n_{s-2}}\subseteq \cdots \subseteq x_{s}x_{s-1}\cdots
x_{1}B_{n_{0}}\text{,}
\end{equation*}%
and hence there is $x_{0}\in B_{n_{0}}\subseteq \mathbb{N}+d$ such that $%
x_{s}x_{s-1}\cdots x_{0}=\tilde{x}$ and $x_{j}x_{j-1}\cdots x_{0}\in
B_{n_{j}}\subseteq C_{t}$ for $0\leq j\leq s$. For $0\leq j\leq i\leq s$,
set $x_{j,i}:=x_{j}x_{j+1}\cdots x_{i}=y_{n_{j-1}+1,n_i}$. We need to show
that the set
\begin{equation*}
A:=\{x_{0,s}\}\cup \{x_{0,j}+f(x_{j+1},\ldots ,x_{i})\mid 0\leq j<i\leq
s,f\in F_{i-j}\}
\end{equation*}%
is contained in $C_{t}$. Fix $0\leq j<i\leq s$ and $f\in F_{i-j}$. We claim
that $x_{0}x_{1}\cdots x_{j}+f(x_{j+1},\ldots ,x_{i})\in B_{n_{j}}$. We have
that%
\begin{eqnarray*}
x_{0,i}+x_{j+1,i}f\left( x_{j+1},\ldots ,x_{i}\right) &\in
&B_{n_{i}}+x_{j+1,i}f\left( x_{j+1},\ldots ,x_{i}\right) \\
&\subseteq &y_{n_{i}}D_{n_{i}}+x_{j+1,i}f\left( x_{j+1},\ldots ,x_{i}\right)
\\
&=&y_{n_{i}}\left( D_{n_{i}}+x_{j+1,i-1}y_{n_{i-1}+1,n_{i}-1}f\left(
x_{j+1},\ldots ,x_{i}\right) \right)
\end{eqnarray*}%
By the definition of $D_{n_{i}}$, we have that%
\begin{equation*}
D_{n_{i}}+x_{j+1,i-1}y_{n_{i-1}+1,n_{i}-1}\big(f(x_{j+1},\ldots
,x_{i})-f(x_{j+1},\ldots ,x_{i-1},y_{n_{i-1}+1}\cdots y_{n_{i}-1})\big)%
\subseteq B_{n_{i}-1}
\end{equation*}%
Hence,%
\begin{equation*}
x_{0,i}+x_{j+1,i}f\left( x_{j+1},\ldots ,x_{i}\right) \in y_{n_{i}}\left(
B_{n_{i}-1}+x_{j+1,i-1}y_{n_{i-1}+1,n_{i}-1}f(x_{j+1},\ldots
,x_{i-1},y_{n_{i-1}+1}\cdots y_{n_{i}-1})\right)
\end{equation*}%
Proceeding in this fashion, after $n_{i}-n_{i-1}$ steps one obtains%
\begin{equation*}
x_{0,i}+x_{j+1,i}f\left( x_{j+1},\ldots ,x_{i}\right) \in x_{i}\left(
B_{n_{i-1}}+f\left( x_{j+1},\ldots ,x_{i-1},1\right) \right) \text{.}
\end{equation*}%
Since by assumption $f\left( x_{j+1},\ldots ,x_{i-1},1\right) =0$, we have
that%
\begin{equation*}
x_{0,i}+x_{j+1,i}f\left( x_{j+1},\ldots ,x_{i}\right) \in x_{i}B_{n_{i-1}}%
\text{.}
\end{equation*}%
Dividing by $x_{i}$ we deduce that
\begin{equation*}
x_{0,i-1}+x_{j+1,i-1}f(x_{j+1},\dots ,x_{i})\in B_{n_{i-1}}\subset
x_{j+1,i-1}B_{n_{j}}.
\end{equation*}%
Therefore, dividing by $x_{j+1,i-1}$ it follows that

\begin{equation*}
x_{0,j}+f\left( x_{j+1},\ldots ,x_{i}\right) \in B_{n_{j}}\subseteq C_{t}
\end{equation*}%
as claimed. This concludes the proof that $A\subseteq C_{t}$.
\end{proof}

\begin{corollary}
\label{Corollary:PR}Fix $d\in \mathbb{Q}$, and $s\in \mathbb{N}$. Then there
exist infinite $\xi _{0}\in {}^{\ast }\mathbb{N}+d$ and $\xi _{1},\ldots
,\xi _{s}\in {}^{\ast }\mathbb{N}_{0}/d+1$ such that, for every for every
function $f:\mathbb{N}^{i}\rightarrow \mathbb{Z}$ such that, for every $%
x_{1},\ldots ,x_{i-1}\in \mathbb{N}$, the function $x\mapsto f\left(
x_{1},\ldots ,x_{i-1},x\right) $ is a polynomial with integer coefficients
vanishing at $1$, one has that
\begin{equation*}
\xi _{0}\sim \xi _{0}\cdots \xi _{i}\sim \xi _{0}\cdots \xi _{j}+f\left( \xi
_{j+1},\ldots ,\xi _{i}\right)
\end{equation*}%
for $0\leq j<i\leq s$.
\end{corollary}

\begin{corollary}
\label{Corollary:PR2}Fix $d\in \mathbb{Q}$. There exist infinite $r\in
{}^{\ast }\mathbb{N}$ and $s\in {}^{\ast }\mathbb{N}_{0}/d$ such that for
every polynomial with integer coefficients $p\left( w\right) $ vanishing at $%
0$ one has that%
\begin{equation*}
r\sim r+p\left( s\right) \sim rs+r+ds\text{.}
\end{equation*}
\end{corollary}

\begin{proof}
By Corollary \ref{Corollary:PR} there exist infinite $r_{0}\in {}^{\ast }%
\mathbb{N}+d$ and $s_{0}\in {}^{\ast }\mathbb{N}_{0}/d+1$ such that%
\begin{equation*}
r_{0}\sim r_{0}+p\left( s_{0}-1\right) \sim r_{0}s_{0}
\end{equation*}%
for every polynomial with integer coefficients $p\left( w\right) $ vanishing
at $0$. Set now $r:=r_{0}-d$ and $s:=s_{0}-1$. Observe that%
\begin{equation*}
r+d\sim r+d+p\left( s\right) \sim rs+r+ds+d
\end{equation*}%
and hence%
\begin{equation*}
r\sim r+p\left( s\right) \sim rs+r+ds
\end{equation*}%
for every polynomial with integer coefficients $p\left( w\right) $ vanishing
at $0$.
\end{proof}

As an application of Theorem \ref{Theorem:PR} and Theorem \ref%
{Theorem:Moreira}, we establish the partition regularity of certain
Diophantine equations.

\begin{example}
\label{Example:PR}Let
\begin{equation*}
p(x,z)=a_{0}x^{d}+a_{1}x^{d-1}z+\cdots +a_{d-1}xz^{d-1}+a_{d}z^{d}
\end{equation*}%
be a homogeneous polynomial of degree $d\geq 1$ in the variables $x$ and $z$%
. Set%
\begin{equation*}
P\left( x,y,z\right) =x^{d}(x-y)+p(x,z)\text{.}
\end{equation*}

The solutions to the equation $P\left( x,y,z\right) =0$ are parametrized by
\begin{equation*}
\begin{cases}
x=r \\
y=r+p(1,s) \\
z=rs%
\end{cases}%
\text{.}
\end{equation*}%
Notice that $p\left( 1,0\right) =a_{0}$ and $p\left( 1,1\right)
=\sum_{i=0}^{d}a_{i}$. Therefore Corollary \ref{Corollary:Moreira} proves
that the equation $P\left( x,y,z\right) =0$ is partition regular when $%
a_{0}=0$, and Corollary \ref{Corollary:PR} proves that the equation $P\left(
x,y,z\right) =0$ is partition regular when $\sum_{i=0}^{d}a_{i}=0$.
\end{example}

\section{Applications\label{Section:equation}}

The goal of this section is to use Theorem \ref{Theorem:necessary} to
establish Theorem \ref{Theorem:equation}. Fix $a,b,c\in \mathbb{Z}$ not all
zero such that $abc=0$ or $a+b+c=0$. We start by considering the polynomial $%
P\left( x,y,z\right) =x^{2}-xy+ax+by+cz=0$. Observe that $\tilde{P}\left(
w\right) =\left( a+b+c\right) w$.

\begin{lemma}
\label{Lemma:minimal-P}If $a+b+c\neq 0$ and $P\left( x,y,z\right)
=x^{2}-xy+ax+by+cz$ satisfies the minimal Rado condition, then one of the
following holds:

\begin{enumerate}
\item $a=b=0$,

\item at most one among $a,b,c$ is zero, and $0\in \left\{
a+b,a+c,b+c\right\} $.
\end{enumerate}
\end{lemma}

\begin{proof}
Suppose that $P$ satisfies the minimal Rado condition, and assume that at
least one between $a,b$ is nonzero. We have that the only root of $\tilde{P}$
is $0$, and the only minimal indices for $P$ have degree $1$. In particular,
any minimal Rado set for $P$ is homogeneous of degree $1$. This implies that
at most one among $a,b,c$ is zero, and $0\in \left\{ a+b,a+c,b+c\right\} $.
\end{proof}

\begin{lemma}
\label{Lemma:equation}Consider the polynomial $P\left( x,y,z\right)
=x^{2}-xy+ax+by+cz$ for $a,b,c\in \mathbb{Z}$ such that $abc=0$ or $a+b+c=0$%
. If $\tilde{P}=0$ or $P$ satisfies the minimal Rado condition, then $P$ is
partition regular.
\end{lemma}

\begin{proof}
By Lemma \ref{Lemma:minimal-P}, either $a=b=0$ or at most one among $a,b,c$
is zero, and $0\in \left\{ a+b,a+c,b+c,a+b+c\right\} $.

When $a=b=0$, $P$ reduces to $x\left( x-y\right) +cz$. Such a polynomial is
partition regular by Example \ref{Example:PR}.

We now consider the case when exactly one among $a,b,c$ is zero. When $c=0$
and $a+b=0$, the equation reduces to $\left( x+a\right) \left( x-y\right) =0$%
, which admits the constant solutions $x=y$, and hence it is (trivially)
partition regular. When $b=0$ and $a+c=0$, the equation reduces to $x\left(
x-y\right) +a\left( x-z\right) =0$, which is partition regular by Example %
\ref{Example:PR}. When $a=0$ and $b+c=0$, the polynomial reduces to $P\left(
x,y,z\right) =x\left( x-y\right) +b\left( y-z\right) $. Observe now that $%
P^{\left( b\right) }\left( x,y\right) =x\left( x-y\right) +b\left(
x-z\right) $, which has just been shown to be partition regular.

Similarly, when $a,b,c$ are all nonzero, and $a+b+c=0$, one has that $%
P^{\left( b\right) }\left( x,y,z\right) =x\left( x-y\right) +\left(
a+b\right) x+cz$. Thus, $P^{\left( b\right) }$ is partition regular, as
shown above. This concludes the proof.
\end{proof}

We now consider the polynomial $P\left( x,y,z\right) =x^{2}-y^{2}+ax+by+cz$.
Observe that $\tilde{P}\left( w\right) =\left( a+b+c\right) w$.

\begin{lemma}
\label{Lemma:minimal-P2}Suppose that $a+b+c\neq 0$ and $P\left( x,y,z\right)
=x^{2}-y^{2}+ax+by+cz$ satisfies the minimal Rado condition. Then at most
one among $a,b,c$ is zero, and $0\in \left\{ a+b,a+c,b+c\right\} $.
\end{lemma}

\begin{proof}
Suppose that $P$ satisfies the minimal Rado condition. We have that the only
root of $\tilde{P}$ is $0$. By Corollary \ref{Corollary:length}, we have
that any minimal Rado set for $P$ is homogeneous of degree $1$. This implies
that at most one among $a,b,c$ is zero, and $0\in \left\{
a+b,a+c,b+c\right\} $.
\end{proof}

\begin{lemma}
\label{Lemma:equation2}Suppose that $a,b,c\in \mathbb{Z}$ are such that
either $a+b+c=0$ or $abc=0$. Set $P\left( x,y,z\right) =x^{2}-y^{2}+ax+by+cz$%
.\ If $\tilde{P}=0$ or $P$ satisfies the minimal Rado condition, then $P$ is
partition regular.
\end{lemma}

\begin{proof}
By Lemma \ref{Lemma:minimal-P2}, either $a=b=c=0$ or at most one among $%
a,b,c $ is zero, and $0\in \left\{ a+b,a+c,b+c,a+b+c\right\} $.

When $a+b=0$, $P\left( x,y,z\right) =0$ reduces to $\left( x-y\right)
(x+y+a)=0$, which admits the constant solutions $x=y$, and hence it is
(trivially) partition regular.

We consider now the case when $a+b\neq 0$ and $a+b+c=0$. By Proposition \ref%
{Proposition:scale}, it suffices to show that the polynomial $P\left(
x,y,z\right) =\frac{1}{4}\left( x^{2}-y^{2}\right) +ax+by+cz$ is partition
regular. Observe that setting $x=r+s$ and $y=r-s$ we have that $P\left(
x,y,z\right) =rs+\left( a+b\right) r+\left( a-b\right) s-\left( a+b\right) z$%
. Therefore, we have that $P\left( x,y,z\right) $ is partition regular if
and only if there exist $r,s\in {}^{\ast }\mathbb{N}$ such that $a+b$
divides $rs+\left( a-b\right) s$ and $r+s\sim r-s\sim \frac{1}{a+b}rs+r+%
\frac{\left( a-b\right) }{a+b}s$. For this, it suffices to find $%
r_{0},s_{0}\in {}^{\ast }\mathbb{N}$ such that $\frac{s_{0}}{\left\vert
a+b\right\vert }\in {}^{\ast }\mathbb{N}$ and $r_{0}+s_{0}\sim
r_{0}-s_{0}\sim r_{0}s_{0}+r_{0}+\frac{a-b}{a+b}s_{0}$, as then one can set $%
r:=\left( a+b\right) r_{0}$ and $s:=\left( a+b\right) s_{0}$. We have that
such $r_{0},s_{0}$ exist by Corollary \ref{Corollary:PR2}.

We now consider the case when $a+b\neq 0$ and $a+b+c\neq 0$. In this case,
we have that $P$ does not satisfy the minimal Rado condition, and hence
there is nothing to prove. To see this, notice that if $P$ satisfies the
minimal Rado condition, then by Lemma \ref{Lemma:minimal-P2} at most one
among $a,b,c$ is zero, and $0\in \left\{ a+b,a+c,b+c\right\} $. Since by
hypothesis $abc=0$ and $a+b\neq 0$, it follows that $c=0$. This gives the
contradiction that $0\in \left\{ a,b\right\} $, concluding the proof.
\end{proof}

Theorem \ref{Theorem:equation} is an immediate consequence of the lemmas in
this section and Theorem \ref{Theorem:necessary}.

\end{document}